\documentclass{lms}
\usepackage{amssymb,amsmath}

\newtheorem{thm}{Theorem}[section]
\newtheorem{cor}[thm]{Corollary}
\newtheorem{prop}[thm]{Proposition}

\newtheorem{lem}[thm]{Lemma}

\newtheorem{remar}[thm]{Remark}

\providecommand{\ld}{,\ldots ,}

\providecommand{\ra}{ \rightarrow }

\providecommand{\se}{ \subseteq }
\providecommand{\lan}{ \langle }
\providecommand{\ran}{ \rangle }

\providecommand{\diag}{\mathop{\rm diag}\nolimits}
\providecommand{\End}{\mathop{\rm End}\nolimits}
\providecommand{\Id}{\mathop{\rm Id}\nolimits}
\providecommand{\Irr}{\mathop{\rm Irr}\nolimits}

\providecommand{\Ker}{\mathop{\rm Ker}\nolimits}

\providecommand{\Stab}{\mathop{\rm Stab}\nolimits}
\providecommand{\Ind}{\mathop{\rm Ind}\nolimits}
\providecommand{\CC}{\mathop{\mathbb C}\nolimits}
\providecommand{\FF}{\mathbb{F}}
\providecommand{\ZZ}{\mathbb{Z}}

\providecommand{\Ste}{{\sf {St}}}

\providecommand{\al}{\alpha}
\providecommand{\be}{\beta}
\providecommand{\ga}{\gamma}

\providecommand{\ep}{\varepsilon}

\providecommand{\lam}{\lambda }

\providecommand{\up}{^{-1}}

\providecommand{\si}{\sigma }
\providecommand{\om}{\omega }

\providecommand{\au}{{automorphism }}

\providecommand{\eis}{{eigenvalues }}

\providecommand{\f}{{following }}

\providecommand{\ho}{{homomorphism }}

\providecommand{\ii}{{if and only if }}

\providecommand{\ir}{{irreducible }}

\providecommand{\irr}{{irreducible representation }}

\providecommand{\itf}{{It follows that }}

\providecommand{\mult}{{multiplicity }}

\providecommand{\rep}{{representation }}
\providecommand{\reps}{{representations }}

\providecommand{\TC}{{\mathbf T}}

\providecommand{\GC}{{\mathbf G}}
\providecommand{\AAA}{{\sf A}}
\providecommand{\SSS}{{\sf S}}

\providecommand{\bl}{\begin{lem}\label }
\providecommand{\el}{\end{lem}}
\providecommand{\med}{\medskip}
\providecommand{\mod}{\bmod \,}

\title[Conjugacy action and the Steinberg square]
{Conjugacy action, induced representations and the
Steinberg square for simple groups of Lie type }

\author[G. Heide, J. Saxl, P. H. Tiep, and A. E. Zalesski]
{Gerhard Heide, Jan Saxl, Pham Huu Tiep, and Alexandre E. Zalesski}

\extraline{The third author
gratefully acknowledges the support of the NSF (grants DMS-0901241 and
DMS-1201374).}
\classno{}
\date{March 14, 2012}

\begin{document}
\maketitle

\begin{abstract} Let $G$ be a finite simple group of Lie type, and let
$\pi_G$ be the permutation representation of $G$ associated with
the action of $G$ on itself by conjugation. We prove that every
irreducible representation of $G$ is a constituent of $\pi_G$,
unless $G=PSU_n(q)$ and $n$ is coprime to $2(q+1)$, where
precisely one irreducible representation fails. Let $\Ste$ be the
Steinberg \rep of $G$. We prove that an \irr of $G$ is a
constituent of the tensor square $\Ste\otimes \Ste$, with the same
exceptions as in the previous statement.
\end{abstract}

\section{Introduction}

The two main problems discussed in this paper do not appear at
first sight to be connected. The first concerns the permutation
representation $\pi_G$  of a finite group $G$ in its action on
itself by conjugation (also called the adjoint action). The second
is concerned with the \ir constituents of the tensor square of the
Steinberg representation of a simple group of Lie type. In
addition, we also deal with the existence of a maximal torus $T$
of $G$ such that every \ir representation of $G$ is  a constituent
of the induced representation $\Ind_T^G(1_T)$.

\medskip
The conjugation action of $G$ on itself, the afforded ${\mathbb
C}G$-module $\Pi_G$ and its character $\pi_G$ are standard objects
of study in the \rep  theory of finite groups. For instance, the
submodule of $G$-invariants of this module forms the center of the
group algebra ${\mathbb C}G$. The multiplicity of the trivial
${\mathbb C}G$-module in  $\Pi_G$ equals the number of conjugacy
classes in $G$. Surprisingly, almost nothing is known about the
multiplicities of other \ir modules in $\Pi_G$. The problem has a
long history; the initial question was to determine those
irreducibles of \mult zero. See \cite{R}, \cite{Sch}, \cite{P} and
the bibliography there. Passman restated this problem to that of
determining the kernel $\Delta_G$ of the \rep ${\mathbb C}G\ra
\End(\Pi_G)$, and studied the problem both for finite and infinite
groups $G$. It becomes clear that the problem cannot be answered
in simple terms for arbitrary finite groups.

Significant progress was achieved in \cite{HE}, where
Passman's problem was solved for finite classical simple groups.
For alternating and sporadic simple groups   $\Delta_G$ was proved
to be zero in \cite{HZ}, but a family of unitary groups $PSU_n(q)$
with $n$  coprime to $2(q+1)$ yields examples with $\Delta_G\neq 0$.
(The first example of this, $PSU_3(3)$, was given by Frame in his
review of \cite{R}.)
For other simple groups the problem remained open,
and is finally solved in this paper. In addition, we determine
$\Delta_G$ in the above exceptional cases.

\begin{thm}\label{main1}
Let $G$ be a finite simple group of Lie type, other than
$PSU_n(q)$ with $n \geq 3$ coprime to $2(q+1)$. Then
$\Delta(G)=0$, that is, every complex irreducible character of $G$
is a constituent of the conjugation permutation character of $G$.

In the exceptional case, $\Delta(G)\cong Mat_m(\mathbb C)$, where
$m=(q^n-q)/(q+1)$, that is,  every complex irreducible character
occurs, except precisely the (unique) irreducible character of
dimension $(q^n-q)/(q+1)$.
\end{thm}

The exception is rather interesting: the unique missing \ir character
is the \ir of smallest dimension greater than $1$, one of the
Weil modules.

\medskip
The second problem we address in this paper concerns the square of
the Steinberg character of a simple group of Lie type. The initial
question was partly motivated by a conjecture of John Thompson,
that each finite simple group possesses a conjugacy class
the square of which covers the group. The natural question then is
whether each finite simple group has an \ir character whose tensor
square contains each \ir of the group. The answer
turns out to be negative in general, see Lemma \ref{other}.
However, we show
that the answer is positive for almost every group, with the same
exceptions as in Theorem \ref{main1}:

\begin{thm}\label{main2}
Let $G$ be a finite simple group of Lie type, other than
$PSU_n(q)$ with $n \geq 3$ coprime to $2(q+1)$. Then every \ir
character of $G$ is a constituent of the tensor square $\Ste^2$ of
the Steinberg character $\Ste$ of $G$.

In the exceptional case, every \ir character occurs, except
precisely the (unique)   
irreducible character of dimension $(q^n-q)/(q+1)$.
\end{thm}

\begin{remar}
{\em 1) The assertion
also holds for the almost simple but imperfect groups
$Sp_4(2)$, ${\!^2G}_2(3)$, ${\!^2F}_4(2)$, and $G_2(2)$.
Moreover, in the first three of these, we can take the
Steinberg character of the derived subgroup to be any
irreducible constituent
of the restriction of Steinberg character of the group;
the assertion remains valid in the simple group.

2) In the exceptional case of Theorem \ref{main2}, the square of
no \ir character contains every \ir of $G$ as a constituent, cf.
Lemma \ref{other}. In other words, in case of groups of Lie type,
the best possible result for the tensor square conjecture is
achieved by our choice of Steinberg module. This also implies that
the assertion fails for the group $G_2(2)'$, in view of its
isomorphism to $PSU_3(3)$.

3) Eamonn O'Brien has checked for us that each sporadic
simple group $G$ possesses an \ir character whose square involves each
irreducible character of $G$ as a constituent, and the same holds
also for the alternating groups of degree at most $17$.
We believe that this is true for the alternating groups in general.}
\end{remar}

\medskip
We obtain two corollaries at this stage.

According to an observation of L. Solomon \cite{So},
the multiplicity of an \ir character in the conjugation character
of $G$ equals the sum of entries in the corresponding row of
the character table of $G$. Hence we have:

\begin{cor}\label{cr2}
Let $G$ be a non-abelian finite simple group,
other than $PSU_n(q)$ with $n \geq 3$ coprime to $2(q+1)$.
The sum of entries in each row of the character table of $G$
is a positive integer.

In the exceptional case, the same is true of each row except for
the second row.
\end{cor}

\begin{cor}\label{cr3}
Let $\GC$ be a simple simply
connected algebraic group, and let $Fr$ denote a Frobenius
endomorphism of $\GC$. Set $G=\GC^{Fr}$ and $H=\GC^{Fr^2}$. Then the restriction
$\Ste^H|_G$ of the Steinberg representation $\Ste^H$ to $G$ equals 
$\Ste^G\otimes \Ste^G$, and hence its \ir constituents
are as described in Theorem $\ref{main2}.$
\end{cor}

Indeed, it was proved by
Steinberg \cite[statement (3), p. 101]{St} that the restriction of
$\Ste^H$ to $G$ coincides with the square of $\Ste^G$. So the
result follows from Theorem \ref{main2}.

\medskip
It follows from the above theorems that for simple groups of Lie type
the Steinberg square character
problem has exactly the same answer as Passman's problem discussed
above. The cause of this coincidence will be seen in Section 5.
In one direction, this is clear: observe that
$\Pi_G=\sum_{M \in \Irr G}M\otimes M^*$.
Consider the module $\Phi_G=\sum_{M \in \Irr(G)}M\otimes
M=\Phi_1\oplus \Phi_2$, where $\Phi_1=\sum_{M\cong M^*}M\otimes M$
and $\Phi_2=\sum_{M\not\cong M^*}M\otimes M$. Then $\Phi_1$ is
isomorphic to a submodule of $\Pi_G$, whereas $\Phi_2$ does not
contains $1_G$ as a constituent. Therefore, if $\Delta_G\neq 0$
then the tensor square
problem has a negative answer. (This does not require $G$ to be
simple.) The converse is more mysterious.

\medskip
It was conjectured in \cite{HZ} that, for every simple group
$G$, $\Delta_G= 0$ \ii there exists a single conjugacy class $C$
such that the permutation module $\Pi_C$ associated with the
action of $G$ on $C$ contains every \ir \rep of $G$ as a
constituent.   The conjecture was proved to be true in \cite{HZ}
for alternating and sporadic simple groups. In this paper we
confirm this conjecture in general, and moreover we prove a  more
general result:

\begin{thm}\label{mn2}
For every  simple group
$G$ there exists a conjugacy class $C$ such that $\Delta_G=
\Delta_C$. Furthermore, if $G$ is a group of Lie type then $C$ can
be chosen semisimple.
\end{thm}

Theorem \ref{mn2}  already implies the Steinberg square result; 
indeed, it follows from Deligne and Lusztig \cite[7.15.2]{DL}
(a result which was suggested by earlier work of Bhama Srinivasan)
that $\phi$ is a constituent of $\Ste^2$ \ii $\phi|_T$ contains
$1_T$ for some maximal torus $T$ of $G$, see Corollary \ref{cck}.

Let $G$ be a quasi-simple groups such that $G/Z(G)$ is a group of
Lie type in defining characteristic $p$. For $\zeta\in \Irr
(Z(G))$ we set $\Irr_\zeta G=\{\phi\in\Irr (G):
\phi|_{Z(G)}=\zeta\cdot \Id\}$.  Recall that the set $\Irr G$ is
the union of $p$-blocks (in the sense of R. Brauer) and the blocks
of positive defect are in bijection with $\Irr (Z(G))$. So
$\Irr_\zeta G$ is a block, unless $\zeta=1_{Z(G)}$; in this case,
$\Irr_\zeta G$ consists of two blocks, the second one is of defect
0.   Then Theorem \ref{mn2} can be generalized as follows:

\begin{thm}\label{mn3}
Let $G$ be a quasi-simple group such that $G/Z(G)$ is a simple
group of Lie type in defining characteristic $p$ with
$(|Z(G)|,p|)=1$, and let $\zeta \in \Irr (Z(G))$. Then there exists
a maximal torus $T$ of $G$ such that $\Irr_\zeta G$ coincides with
the set of all \ir constituents of ${\rm Ind}_T^G (\tau)$ for any
$\tau\in\Irr (T)$ above $\zeta$, unless $G=SU_n(q)$,
$(2(q+1),n)=1$.

In the exceptional case $ {\rm Ind}_{T}^G (\tau)$ may not contain
a single nontrivial \ir \reps of $G$ which has degree
$(q^n-q)/(q+1)$ or $(q^n+1)/(q+1)$, and this happens precisely for
$q^2-1$ irreducible representations $\tau$ of $T$. In particular,
the only irreducible character of $G$ which is not a constituent
of $\Ind^G_T(1_T)$ is the (unique) unipotent character of degree
$(q^n-q)/(q+1)$.
\end{thm}

In fact, the torus $T$ can be chosen so that the group $T/Z(G)$ is
cyclic, and in Theorem \ref{mn2}  $C_G(c)$ is cyclic for every $c$
in the class $C$. By Frobenius reciprocity, $\phi$ is a
constituent of ${\rm Ind}_T^G (\tau)$ \ii $\phi|_T$ contains $\tau$.
Thus, Theorem \ref{mn3} is equivalent to the statement that there
is a maximal torus $T$ such that, for every non-trivial \irr
$\varphi$, the restriction $\varphi|_T$ contains all $\tau\in\Irr(T)$ 
satisfying $\tau|_{Z(G)}=\zeta$. In this form, the result
probably  remains valid for cross-characteristic representations
too, and  we prove this for classical groups (see Theorem
\ref{mn7}). This is partially based on a result concerning the
action of some parabolic subgroup on its unipotent radical:

\begin{prop}\label{ab1}
Let $G$ be a quasi-simple classical group over a field $\FF_q$
of characteristic $p$. Suppose that $G$ is none of the following
groups: $Sp_4(3)$ for $q = 3$, $PSU_{n}(q)$, $n$ odd and divisible by $q+1$,
and $\Omega_{2n+1}(q)$, $(q-1)n/2$ even. Then there exists a
$p$-group $A$ and a cyclic self-normalizing subgroup $C\leq
N_G(A)$ such that every $N_G(A)$-orbit on $A\setminus \{1\}$
contains a regular $C$-orbit.
\end{prop}

In Proposition \ref{ab1} the orbits in question are considered with
respect to the conjugation action, and a $C$-orbit is called
regular if its size equals $|C|$.

\medskip

In the remainder of the paper we prove the above results.
Section 2 contains a number of preliminary observations. In Section 3 we prove
Theorem \ref{mn3}  for  classical  groups. This part is partially
based on the thesis \cite{HE} of the first author, written under
the supervision of the fourth author. In Section
\ref{exceptional}, the exceptional groups of Lie type are
considered.

Our method is different for classical and exceptional groups of
Lie type. In the former case a key tool is  in the study of the
``internal permutation module''
for the conjugacy action of a parabolic subgroup on its unipotent radical.
In the latter case, our arguments rely on certain character estimates.

\med
{\bf Notation.} The greatest common divisor of integers $m,n$ is
denoted by $(m,n)$. 
 Let $\FF_q$ be the field of $q$
elements, where $q$ is a power of a prime $p$. If $\FF$ is a
field, $\FF^\times$ denotes the multiplicative group of $\FF$. The
identity $n\times n$-matrix is denoted by $\Id$. The
block-diagonal matrix with blocks $A,B$ is denoted by
$\diag(A,B)$, and similarly for more diagonal blocks.

If $V$ is an orthogonal, symplectic or unitary space then $G(V)$
is the group of all linear transformation preserving the unitary,
symplectic or orthogonal structure  on $V$.  Our notation for
classical groups is standard; in particular, if $V$ is orthogonal
space then $\Omega (V)$ is the subgroup of $SO(V)$ of elements of
spinor norm 1. In the unitary case we denote by $\sigma $ the \au
of $GL(V)$ extending the involutory \au of the ground field, in
other cases, for uniformity, $\sigma$ is assumed to be the
identity \au of $GL(V )$.

 For a group $G$
we denote by $Z(G)$ the center of $G$, and for a subset $X$ of $G$
we write $C_G(X)$ for the centralizer and $N_G(X)$ for the
normaliser of $X$ in $G$. We use the symbol $1_G$ to denote the
trivial \rep of $G$ of degree 1 or the trivial one-dimensional
$\FF G$-module. For a subgroup $H$ of $G$ and an $\FF G$-module
$M$ (resp., a representation $\phi$) we write $M|_H$ (resp.
$\phi|_H$)  for the restriction of $M$ (resp., $\phi$) to $H$. To
say that $M$ (resp., $\phi$) is  an \ir $\FF G$-module (resp., \ir
$\FF$-representation of $G$) we write  $M\in \Irr _{\FF}(G)$
(resp., $\phi\in\Irr_{\FF}(G)$). If $\FF=\CC$, the complex number
field, we usually drop the subscript $\CC$. If $M$ is an
irreducible $\FF G$-module, then $Z(G)$ acts as scalars on $M$. By
a {\it central character} of $M$ we mean the linear character of
$Z(G)$ obtained from this scalar action of $Z(G)$ on $M$.

\section{Preliminaries}

Let $G$ be a finite group, let $\pi$ be the permutation character
of $G$ in its conjugation action on itself. The orbits are
the conjugacy classes of $G$. It follows that $\pi$ is the sum of
permutation characters of $G$ on its conjugacy classes. Using
Frobenius reciprocity, we have the following lemma
(cf. Problem 1.1 of \cite{HZ}):

\begin{lem}\label{21}
An \ir character $\chi$ of $G$ is a constituent of 
$\pi$  if and only if for some $g \in G$, the
restriction of $\chi$ to $C_G(g)$ contains the principal character
of $C_G(g)$.
\end{lem}

Hence to prove Theorem \ref{main1} for classical groups, it is sufficient to
prove:

\begin{thm}\label{111} Let $G$ be a finite simple classical group,
other than $PSU_n(q)$ with $n \geq 3$ coprime to $2(q+1)$. Then
there exists a cyclic self-centralizing subgroup $T$ of $G$ such
that, for every \irr $\phi$ of $G$,  the trivial representation
$1_T$ is a constituent of the restriction $\phi|_T$, with the
exception described in Theorem $\ref{main1}$. In the exceptional
case, the same is true for all but one $\phi$, which is the
irreducible Weil \rep of dimension $(q^n-q)/(q+1)$.
\end{thm}

The exceptional cases were discussed in the earlier paper \cite{HZ}
of two of the authors, and are discussed further in Section 3.5 below.

\begin{lem}\label{ab2}
 Let $G$ be a finite group with  subgroups  $A,B$ such that $A$ is abelian
 and $B \leq N_G(A)$. Let $\FF$ be an algebraically closed field of
characteristic $\ell$, where either $\ell = 0$ or $\ell > 0$ is
coprime to $|A|$, and let $M$ be an $\FF G$-module. For an
irreducible $\FF$-\rep $\al$ of $A$ set $M_\al := \{m\in M \mid
am=\al(a)m \mbox{ for all }a\in A\} $.

 Suppose that $M_\al  \neq 0$, and moreover the group
 $C := \Stab_B(\al)$ has a common eigenvector  $v\in M_{\al}$, that is, $cv
= \gamma(c)v$ for all $c \in C$ and some linear $\FF$-character
$\gamma$ of $C$.\\ Then $M|_B$ contains the induced module
$\Ind^B_C(\gamma)$. In particular, if $C = 1$, then $M|_B$
contains a copy of the regular $\FF B$-module.
\end{lem}

\begin{proof}
By Maschke's theorem, $M = \oplus_{\alpha}M_{\alpha}$, and, since
$B \leq N_G(A)$, $B$ permutes the summands of this decomposition
in accordance to its action on the set of all irreducible
$\FF$-\reps  of $A$. Decompose $B = \cup^{n}_{i=1}b_iC$ as a
disjoint union of $C$-cosets (with $b_1 = 1$), and set $v_i =
b_iv$. Then the vectors $v_i$ belong to pairwise distinct
$A$-eigenspaces $M_{b_i(\al)}$, in particular, they are linearly
independent. Now $B$ acts on $N := \oplus^n_{i=1}\FF v_i$, and the
$B$-module $N$ is isomorphic to $\Ind^B_C(\gamma)$.
\end{proof}


Let $r$ be a $p$-power, and let $\sigma_0$ be either trivial \au
of  $\FF_r$ or $r=q^2$ and $\sigma_0$ denote the \au of $\FF_r$ of
order 2. We denote by $\sigma$ the \au induced by $\sigma_0$ on
$GL_n(r)$ and on the algebra of $n\times n$-matrices over
$\FF_r$ for any integer $n>0$.

\begin{lem}\label{8h10}
{\rm (i)} Let $G=GL_n(q)$ and let $s\in G$ be an \ir element of order
$q^n-1$. Then $s$ is conjugate to $s\up$ in $G$ \ii $(n,q)\in
\{(2,2),(1,2),(1,3)\}$.

{\rm (ii)} $s^2$ is conjugate to $s^{-2}$ in $G$ \ii $n\leq 2$ and
$q\leq 3$.

{\rm (iii)} Let $r=q^2$, $G=GL_n(r)$, and let $s\in G$ be
an \ir element of order $r^n-1$. Let  
 $\sigma_0$ be the Galois \au of $\FF_r/F_{q}$. Let
$\sigma$ be the \au of $G$ induced by $\sigma$ in the natural way.
Set $s_1=s^{q+1}$. Then $s_1$ is conjugate to $\sigma (s_1\up )$
\ii $n=1$ and $q\leq 3$.\end{lem}


\begin{proof}
(i) $s$ is \ir in $GL_n(q)$ (otherwise $|s|$ divides $q^i-1$ for
some $i<n$). By Schur's lemma, the enveloping $\FF_q$-algebra $K$
of $\lan s\ran$ is a field. Let $V$ be the natural module for $G$.
Then $V$ is an \ir $K$-module, hence $|V|=|K|=q^n$ and  $s$ is a
generator of the multiplicative group $K^\times$ of $K$.

If $n=1$, the claim is trivial, so we assume $n>1$, hence also
$|s|>2$. Suppose that $s\up =gsg\up$ for some $g\in G$. Then
$gKg\up =K$, so $g$ induces a field \au on  $K$ of order 2 (as
$s\neq s\up$).   Let $L=C_K(g)$. By Galois theory,  $K:L=2$, and
hence $|L|=q^{n/2}$. Then $|K^\times/L^ \times|=q^{n/2}+1$ and
$s^{q^{n/2}+1}=t$ is a generator of $L^\times$. On one hand,
$gt=tg$ as $t\in L$. On the other hand, $gtg\up =t\up$ as $t$ is a
power of $s$. Therefore, $t=t\up$ and $t^2=1$. It follows that
$|L|=2$ or $ 3$. As $\FF_q\cdot \Id\se L$, one observes that
$\FF_q=L$ so $q=2$ or 3 and $|V|=|K|=4$ or 9. Therefore, $n=2$,
and the lemma follows.

\medskip
(ii) $s^2$ is \ir unless $s^2$ belongs to a proper subfield of
$F_{q^n}$, which is not the case. So  $\lan s^2\ran =K$. Then, as
above, $s^{-2}=g s^2g\up$ implies $K:L=2$, where $L=C_K(g)$. Then
the generator  $t$ of $L^\times$ is a power of $s^2$, so again
$|K|=4$ or  $9$.

\medskip
{\rm (iii)} It is easy to check that $s_1^2\neq 1$. As above, $s$ is
\ir in $GL_n(r)$, and $K$, the enveloping algebra of $\lan s\ran$
over $F_r$ is a field.  Suppose that $\sigma (s_1\up )=gs_1g\up$
for some $g\in G$ and let $\tau$ be the \au of $G$ defined by
$\tau (x)=\sigma (gxg\up )$. Then $\tau (s_1)=s_1\up$ hence $\tau
(K) =K$ and $\tau^2|_K=\Id$. Let $L=\{a\in K:\tau (a)=a\}$. Then
$K:L=2$. Let $t\in L$ be the generator of $L$. Then $t=\tau
(t)=t\up$ as $t$ is a power of $s$ (because $s_1\notin L$). Hence
$t^2=1$ and $|L|\leq 3$.
 As $F_{q}\cdot \Id\se L$, we observe that $F_{q}\cong L$, so
 $q=2$ or 3.  As
 $\FF_r:F_{q}=2=K:L$ and $F_{q}\cong L$, we observe that $K=\FF_r$, and
  hence $n=1$.
 The converse is obvious.
\end{proof}

\med
Let $q,r,\sigma$ be as above. Let $M=Mat_n(\FF_r)$ be the matrix
algebra, and $G=GL_n(r)$. For $x\in G$ consider the mapping $M\ra
M$ defined as $m \mapsto xm\si(x^t)\up$ for $m \in M$. This
mapping is linear, and hence $M$ becomes an $\FF_qG$-module. Let
$\ep=\pm 1$, and assume $\ep=1$ if $\si\neq 1$ or $p=2$. Set

\begin{equation}\label{eqdef}
 L=
\{x\in M:\si(x^t)=-\ep x\},\mbox{ if }p\neq 2 \mbox{ or }\sigma\neq 1,~
\end{equation}

\begin{equation}\label{eqdef2}
L_0=\{x\in M:x^t= x \mbox{ and }x~ \mbox{ has zero diagonal}\},
~{\rm if}~p=2~{\rm and}~\si=1.
\end{equation}

\noindent Observe that the subspaces $L$ and $L_0$ are $G$-stable.
The action of $G$ on $Mat_n(\FF_r)$, as well as on $L$ and $L_0$,
is referred below as the congruence action. This induces the dual
action on the character groups $\Irr(L)$. In particular, $L$ and
$\Irr(L^+)$ are dual $\FF_rG$-modules, as well as $L_0$ and $\Irr(L_0^+)$, where $L^+,L_0^+$ are the additive groups of $L,L_0$,
respectively. Let $V$ be the natural module for $G=GL_n(r)$. If
$p=2$ and $\si=1$ then $L_0$ is isomorphic to the module of the
alternating forms on $V$, and the dual of $L_0$ is isomorphic to
the module of the alternating forms on the dual of $V$ (cf.
\cite[Lemma 2.14]{VZ}). In addition, the dual of $L$ is isomorphic
to the module of the quadratic  forms on $V$ (cf. \cite[Lemma
2.16]{VZ}).

\begin{lem}\label{dd4}
The kernel of the congruence action of
$G=GL_n(r)$ on $L$ coincides with $D_\si=\{g\in Z(G):\si(g)g=1\}$.
If $n>2$ then $D_\si$ is also the kernel of the action of $G$ on
$L_0$. In addition,  $D_\si=U_1(r)\cdot \Id$ if $V$ is unitary,
otherwise $D_\si=\{\lam \cdot \Id:\lam\in \FF_q,\lam^2=1\}$.\el

\begin{proof} This is straightforward.\end{proof}

\begin{lem}\label{p87}
Let  $G'=SL_n(r)$ and $L,L_0$ as above. Then $G'$ fixes no
non-zero element of $L$ and of its dual module $\Irr(L)$, unless
$n=2$ and $\si =1$, or $n=1$. The same is true for $L_0$ and
$\Irr(L_0)$. In addition, a similar statement is true for a
subgroup $X< SL_4(2)$ isomorphic to the alternating group $\AAA_7$.
\end{lem}

\begin{proof} The  claim for $L$ is \cite[Lemma 3.4]{DZ}.
If $L$ is \ir then so is the dual $\Irr(L)$, which implies  the
claim for $\Irr(L)$, and similarly, for $\Irr(L_0)$. Similar
argument is valid for $L_0$. For the additional statement, if
$X\cong \AAA_7$ fixes an alternating form $f$ then $f$ is
non-degenerate as both $V\cong \FF_2^4$ and its dual are \ir
$\FF_2X$-modules. Then $X$ is isomorphic to a subgroup of
$Sp_4(2)$, which is false.
\end{proof}

\begin{lem}\label{sz1} {\rm \cite[Theorem 1.1]{SZ}}
Let $PSL_n(r)\leq X\leq PGL_n(r)$ where $(n,r)\neq (2,2),(2,3)$
and let $C$ be a cyclic subgroup of $X$.  Let $\Pi$ be a
permutation $X$-set, on which  $PSL_n(q)$ acts  non-trivially.
Then one of the \f holds:

{\rm (i)} $C$ has a regular orbit  on $\Pi;$

{\rm (ii)} $X=SL_4(2)\cong \AAA_8$, $|C|=15$ or $6$ and $|\Pi |=8;$

{\rm (iii)} $X=PGL_2(5)\cong \SSS_5$, $|C|=6$ and $|\Pi |=5$.
\end{lem}


We remark that the case $|C|=6$ in (ii) is missing in the
original statement  in \cite[Theorem 1.1(case b)]{SZ}.
This omission has no effect on any other result in \cite{SZ}.

\medskip
In what follows, we define a Singer subgroup of $SL_n(r)$  to be
an \ir subgroup of order $(r^n-1)/(r-1)$.

\begin{lem}\label{p44}
Let $G'=SL_n(r)\leq G\leq GL_n(r)$, $n>1$, let $S$ be a Singer
subgroup in $G'$, and $L$
 the $\FF_rG$-module defined above.
  Let $\al$ be a non-trivial
character of the additive group of $L$ (resp., $L_0$). Then there
exists  $\beta\in G'\al$ such that $C_S(\beta)=D_\sigma$
except for the case where 
$n=2$ and $\sigma = 1$.
\end{lem}

\begin{proof}
In the  action of $G$ on $\Irr(L)$ an element $g\in G$ sends
$\al\in \Irr(L)$ to $g(\al)$, where $g(\al)(l):=\al(glg\up )$ for
all $l\in L$. By Lemma \ref{p87}, $G'$ fixes no non-trivial
character $\al\in \Irr(L)$, except for the case with 
$n=2$,  $\sigma = 1$.

In the non-exceptional case $G'\al\neq \al$. Set
$G_0=\cap_{\gamma\in G\al} C_G(\gamma)$. Then $G_0 $ is normal in
$G$, and hence $G_0< Z(G)$ unless  $(n,r)=(2,2)$ or $(2,3)$. By
Lemma \ref{sz1}, there exists $\beta\in G\al$ such that for $t\in
S$ either $t\in G_0$  or $t\beta\neq \beta$, unless possibly
$n=4,r=2$ or $n=2,r=5$. The latter case appears in the conclusion
of the lemma. In the former case, $G'=G=SL_4(2)$, the orbit $G\al$
is of size 8 and $C_{G}(\al )\cong \AAA_7$. In this case the result
follows from Lemma \ref{p87}.

Observe that $s\beta=\beta$ for $s\in Z(G)$ implies $s\in
D_\si\cap G$. Indeed, let $s=\lam \cdot \Id$. Then
$sl\si(s)=\lam\si(\lam)l$, and then $\al
(sl\si(s))=\lam\si(\lam)\al(l)$ for all $l\in L$. Therefore,
$s\al=\lam\si(\lam)\al$. \itf $s\al=\al$ implies
$\lam\si(\lam)=1$, and the claim follows from Lemma \ref{dd4}.

The lemma follows from this observation, including the case $n=1$.
\end{proof}

The \f observation  is a slight modification of \cite[Theorem
1.7]{HZ}, where it is assumed that $Z(G)=1$.

\begin{lem}\label{hz1}
Let $G$ be a finite group with cyclic Sylow $p$-subgroup $P$.
Suppose that $(p,|Z(G)|)=1$ and $C_G(g)=Z(G)P$ for $g\in P$ of
order $p$. Let $\phi$ be an \irr of $G$ over the complex numbers,
and let $\phi|_{Z(G)}=\zeta\cdot \Id$, where $\zeta\in\Irr
(Z(G))$. Then either $\dim\phi<|P|$ or $\phi$ is a constituent of
$(\zeta\times \lam_P)^G$ for every $\lam\in\Irr(P)$.\end{lem}

\begin{proof} Assume that $\dim\phi\ge |P|$.
By \cite[Lemma 3.1]{HZ}, $\phi|_P$ contains every \irr of $P$ as a
constituent. Therefore, $\phi|_{PZ(G)}$ contains every \irr $\lam$
of $PZ(G)$ such that $\lam|_{Z(G)}=\zeta$. So the result follows
by Frobenius reciprocity.
\end{proof}

\section{Classical groups}

In this section $q$ is a power of a prime $p$ and $\FF$ is an
algebraically closed field of characteristic $\ell \neq p$.
We denote by $V$ a finite-dimensional vector space over a finite
field, endowed by a structure of a non-degenerate unitary,
symplectic or orthogonal space, and we denote by $G(V)$ the
group of all elements of $GL(V)$ preserving the structure. In this
section we prove

\begin{thm}\label{mn7}
Let $G$ be a
quasi-simple group such that $G/Z(G)$ is a simple classical group
 in defining characteristic $p$ and $(|Z(G)|,p|)=1$. Then
there exists a maximal cyclic torus $T$ of $G$ such that whenever
$\phi$ is a non-trivial \ir $\FF$-\rep of $G$ with central
character $\zeta$, every $\tau\in\Irr_{\FF} T$ is a constituent of
$\phi|_T$, unless $G=SU_n(q)$, $(2(q+1),n)=1$, and
$\zeta=1_{Z(G)}$ and $\phi= \phi_{min}$, where $\phi_{min}$ the
unique \ir representation of $G$ of degree $(q^n-q)/(q+1)$.
\end{thm}

If $G$ in Theorem \ref{mn7} is simple then there exists a cyclic
self-centralizing subgroup $C$ of $G$ such that
$\phi|_C=\rho^{reg}_C+\xi$ for some proper \rep $\xi$ of $C$,
where $\rho^{reg}_C$ denotes the regular \rep of $C$.

Throughout this section $G$ satisfies the hypothesis of Theorem
\ref{mn7}, and in every subsection we specify the group $G/Z(G)$
to be considered.

\subsection{The groups $G$ with $G/Z(G)=P\Omega^-_{2n}(q)$, $n>3$,
 or $PSL_n(q)$, $n>1$, $(n,q)\neq (2,2),(2,3)$}

In this section $G/Z(G)$ is one of the above groups,
 so $G/Z(G)$ is simple.
These groups are easier to handle than the remaining  ones.  Our
reasoning is practically identical for both of them, but the
orthogonal group case  requires more attention to detail.

Let $V$ be an orthogonal space of dimension $2n$ over $\FF_q$ of
Witt index $n-1$. Let $W$ be a 1-dimensional singular subspace of
$ V$. Then we choose a basis $b_1\ld b_{2n}$ of $V$ such that
$b_1\in W$, $b_2\ld b_{2n-1}\in W^\perp$, and the Gram matrix of
$V$ is of shape $\begin{pmatrix}0&0&1\cr 0&\Gamma&0\cr
1&0&0\end{pmatrix}$, where $\Gamma$ is a symmetric matrix of size
$2n-2$. (If $p=2$, the Gram matrix is obtained from the
corresponding symplectic structure on $V$.) Set $W_0=\lan b_2\ld
b_{2n-1}\ran_{\FF_q}$. Then $W_0$ is a non-degenerate orthogonal
space of dimension $2n-2$ and of Witt index  $n-2$.

\med
 Set $$G_1=SO^-_{2n}(q),~~~~~~~~~~~~~~~  H := \{g\in G_1 \mid g\mbox{ stabilizes }W,~W_0, \mbox{ and }
  \lan b_{2n}\ran_{\FF_q}\}$$ and
$$A := \{g \in G_1 \mid gb_1=b_1, \mbox{ and }g\mbox{ acts
trivially on }
  W^\perp/W\},$$
 It is well known that $A$ is an abelian group and $A <  \Omega^-_{2n}(q)$.

 One observes that, with respect to the
above basis, $H$ and  $A$ can be described, respectively, as the
sets of matrices  $$\begin{pmatrix}\lam\up&0&0\cr 0&h&0 \cr
0&0&\lam\end{pmatrix}~~{\rm and} ~~
\begin{pmatrix}1&\Gamma(w)^t&0\cr 0&\Id_{2n-2}&w\cr
0&0&1\end{pmatrix} ,$$ where $w\in W_0 $,  $ \lam\in F^\times_q$
and $h\in SO(W_0)$. The action of $H$ on $A$ can be described
 in terms of $W_0$ as follows. Let $g=\diag(\lam,
h,\lam\up)\in H$ and let $w$ corresponds to $a\in A$ as above.
Then $gag\up$ corresponds to $\lam h(w)$. In particular, if
$gag\up =a$ then $ h(w)=\lam\up w $.

It is well known that $SO(W_0)\cong SO^-_{2n-2}(q)$ contains an
\ir cyclic subgroup $S$ of order $q^{n-1}+1$, called a Singer
subgroup
 in \cite{H2}.
Let $T_1$ be a subgroup of $H$ consisting of  matrices
$\diag(\lam, h,\lam\up)$ with $h\in S$. Note that $Z(G_1) <
T_1$. 
If $q$ is even then $(q-1,q^{n-1}+1)=1$, and hence $T_1\cong
T_1/Z(G_1)$ is cyclic. We observe that $ T_1/Z(G_1)$ is cyclic for
$q$ odd too. Indeed, as $(|\FF^\times_q|, |S|)=(q-1,q^{n-1}+1)=2
$, it follows that the projection of $T_1$ into $G_1/\{\pm
\Id_{2n}\}$ is a cyclic group.

\med
If $G_1=SL_n(q)$ then we use $V$ to denote the natural
module, and let $P$ be the stabilizer of an $(n-1)$-dimensional
subspace $W_0$ of $V$. Define $H,A$ as the subgroup of $P$ formed,
respectively,  by the matrices of the shape:

$$\begin{pmatrix}h&0\cr 0&\det (h\up) \end{pmatrix} ~~{\rm and}~~
\begin{pmatrix}\Id_{n-1}&w\cr 0&1\end{pmatrix} ,$$
where $w\in W_0 $ and $h\in GL_{n-1}(q)$.  Then $A=O_p(P)$ and
$P=HA$. Obviously,  $H\cong GL_{n-1}(q)$.
 Let $T_1$ be a subgroup
of $H$ of order $q^{n-1}-1$, which corresponds to a Singer
subgroup $S$ under an isomorphism  $H\ra GL_{n-1}(q)$. Again,
$Z(G_1) < T_1$.

\begin{lem}\label{ca1}
For every $1\neq a\in A$, $C_{T_1}(a)$ consists of scalar
matrices. In other words, $A\setminus \{1\}$ is a union of regular
orbits for the quotient group $T_1/Z(G_1)$.
\end{lem}

\begin{proof} Suppose the contrary. Let $g\in C_{T_1}(a)$, $a\neq 1$.
 Note that $S$ is \ir on
$W_0$, and hence is contained in a Singer subgroup $S'$, say, of
$GL(W_0)$. It is well known that if $s\in S$
has an  eigenvector on $W_0$ then $s$ is scalar. We have to show
that $g$ is scalar.

Suppose $G_1=SL_n(q)$. Then $g\in C_{T_1}(a)$ implies
 $(\det h)\cdot h\cdot w=w$ for some $w\neq 0$, This means
  that $w$ is an
eigenvector for $h$ with eigenvalue $\det h\up$. By the above, $h$
is scalar, so $h=\det h\up \cdot \Id_{n-1}$, and hence $g=\det
h\up \cdot \Id_n$ is scalar.

Suppose $G_1=SO^-_{2n}(q)$. Similarly, $g\in C_{T_1}(a)$ implies
 $\lam\up\cdot h\cdot w=w$ with $w\neq 0$.   Therefore, $h$ is
scalar, and hence $h=\lam\cdot \Id_{2n-2}$. As $h\in SO(W_0)$, we
have $\lam=\pm 1$, which implies that $g$ is scalar.

It follows that the conjugation action of $T_1$ on $A$ partitions
$A\setminus \{1\}$ as a union of regular orbits for the group
$T_1/Z(G_1)$.
\end{proof}

If $G/Z(G)\cong PSL_n(q)$ then $G/Z(G)\cong G_1/Z(G_1)$. If
$G/Z(G) \cong P\Omega_{2n}^-(q) $ then $G/Z(G)$ is isomorphic to a
subgroup of index at most 2 in  $G_1/Z(G_1)$. Set
$\overline{T}=T_1/Z(G_1)$ in the former case, and
$\overline{T}=(T_1/Z(G_1))\cap (G/Z(G))$ in the latter case. As
$Z(G_1)$ is the kernel of the action of  $T_1$ on $A$ via
conjugation, every $T_1/Z(G_1)$-orbit on
$A\setminus \{1\}$ is regular. Therefore, the same is true for
every  subgroup of $T_1/Z(G_1)$, in particular for $\overline{T}$.

Recall that the  central character of $M$ is the linear
character of $Z(G)$ obtained from the scalar action of $Z(G)$ on
$M$.

\begin{prop}\label{ca2} Let $G/Z(G)$ be $PSL_n(q)$,
$(n,q)\neq (2,2), (2,3)$, or $P\Omega^{-}_{2n}(q)$, $n>3$. Let $M$
be a non-trivial \ir $\FF G$-module. Let $T$ be the preimage in
$G$ of the group  $\overline{T}$,   and let $\zeta$ be the central
character of $M$. Then $M|_{T}$ contains the induced module
$\Ind^T_{T \cap Z(G)}(\zeta)$. In particular, $M|_T$ contains
$1_T$ if $\zeta=1_{Z(G)}$.
\end{prop}

\begin{proof}
Let $A$ be as in Lemma \ref{ca1}. As $|Z(G)|$ is coprime to $p$,
$G$ contains a $T$-invariant subgroup $B\cong A$ such that the
conjugation action of $T$ on $B$ is the same as $T_1$ on $A$, that
is, $A,B$ are isomorphic $\overline{T}$-sets.  Then $B\setminus
\{1\}$ is a union of regular   $T/(T\cap Z(G))$-orbits (due to
Lemma \ref{ca1}). Since the actions of $T/(T \cap Z(G))$ on $B
\setminus \{1\}$ and on $\Irr(B) \setminus \{1_B\}$ are dual to
each other, we see that the set $\Irr(B)\setminus \{1_B\}$
 is a union of regular $T/Z(G)$-orbits. Then apply Lemma \ref{ab2}.
\end{proof}

\subsection{Symplectic groups}

For symplectic groups $G=Sp_{2n}(q)$ the result can be easily
deduced from that for $H:=SL_2(q^n)$. This is a special case of
Proposition \ref{ca2}.

It is well known that there is an embedding of $e~:~H\ra G $. Let
$T$ be the subgroup of $H$ of order $q^n-1$ (one can describe
 $T$ as the matrix group $\{\diag(a,a\up): a\in \FF_{q^n}^\times
 \}<  SL_2(q^n)$.

\begin{prop}\label{sp99} Let $G=Sp_{2n}(q)$, $n> 1$, $(n,q)\neq (2,2),(2,3)$.
Let $C=e(T)$. Let $M$ be a non-trivial \ir $\FF G$-module with
central character $\zeta$. Then $C$ is self-centralizing and
$M|_C$ contains the induced module $\Ind^C_{Z(G)}(\zeta)$.
\end{prop}

\begin{proof}  The result follows from Proposition \ref{ca2} for
$SL_2(q^n)$ as long as we show that $e(T)$ is self-centralizing in
$G$. It is well known that $e(T)$ is a maximal torus for $G$,
which implies the claim, provided the torus is self-centralizing.
However, we provide an elementary direct argument for this.

Note that the embedding $e$ is obtained via a vector space
embedding $\FF_{q^n}^2\ra \FF_{q}^{2n}$, induced from the
embedding $\FF_{q^n}\ra  \FF_{q}^{n}$ when $\FF_{q^n}$ is viewed
as a vector space over $\FF_q$. \itf $e(T)$ is reducible, and in
fact a sum of two \ir $n$-dimensional $\FF_q e(T)$-submodules
$W_1,W_2$, say. They are totally isotropic as $Sp_{2n}(q)$
contains no \ir element of order $q^n-1$. In this situation it is
a standard fact that  $V$ has a basis $\{b_1\ld b_{2n}\}$ such
that $b_i\in W_1$, $b_{n+i}\in W_2$ for $i=1\ld n$, and
$(b_i,b_{n+j})=\delta_{ij}$. Then the representations of $T$ on
$W_1,W_2$ are dual to each other. As $T$ is cyclic, the dual \rep
is obtained via the automorphism $t\ra t\up$ ($t\in T$). They are
non-equivalent unless $n=2,q=2$, see Lemma \ref{8h10}.  It follows
from Schur's lemma that $C_G(e(T))=e(T)$, as claimed.
\end{proof}

Note that $Sp_4(2)$ is not simple, and $PSp_4(3)\cong SU_4(2)$.
Proposition \ref{sp99} is not true for $Sp_4(3)$, but Theorem
\ref{111} is true due to this isomorphism and Proposition
\ref{872} below. Note that $Sp_4(3)$ (but not $PSp_4(3)$) is an
exception also for Proposition \ref{ab1}.

\subsection{Groups $G$ with $G/Z(G)=P\Omega^+_{2n}(q)$, $n>3$,
 or $PSU_{2n}(q)$, $n>1$}

 Set $r=q$ if $G/Z(G)=P\Omega^+_{2n}(q)$,
and $r=q^2$ if $G/Z(G)=PSU_{2n}(q)$.
Let $V$ be a unitary space of dimension $2n>2$
over $\FF_r$, or an orthogonal of dimension $2n>6$ and Witt index
$n$ over $\FF_r$. This means that $V$ has a totally isotropic subspace
$W$ of dimension $n$ (totally singular, if $V$ is orthogonal and
$r$ is even).

Let $B=\{b_1\ld
b_{2n}\}$ be a basis of $V$ such that $b_1\ld b_n\in W$
 and the Gram matrix $\Gamma$ corresponding to $B$ is of shape $$\Gamma
=\begin{pmatrix}0&\Id_n\cr \Id_n&0\end{pmatrix}, $$ Then
$V=W\oplus W_1$, where   $W_1=\lan b_{n+1}\ld b_{2n}\ran$ is a
totally isotropic (singular) subspace of $V$. If   $B$ is fixed,
 $G(V)$ can be described as $\{g\in GL(V): g\Gamma \sigma (g)^t =
\Gamma\}$, except when $V$ is orthogonal and $p=2$.

Let $P$ be the stabilizer of $W$ in $G(V)$, and $H=\{ g\in P:
gW_1=W_1\}$. Under the basis $B$, the matrices of $H$ are of shape
$\diag(h,\si (h^t)\up )$, where $h\in GL_n(r)$. If $V$ is
orthogonal and  $q$ is odd then $H$ is isomorphic to a subgroup of
order 2 in $GL_n(r)$, see \cite[Lemma 2.7.2]{KL90}. If $V$ is
unitary then the determinant condition $\det h\si (h^t)\up =1$
implies $\si (\det  (h))=\det(h)$, whence $\det(h)\in U_1(q)$.
Therefore, $H$ is isomorphic to a normal subgroup of $GL_n(r)$ of
index $q+1$. In other cases $H\cong GL_n(r)$.

Set $X=\diag(x,\sigma ( x^t)^{-1})$, where
$x\in GL(W)$ is \ir of maximal order subject to the condition
that $X\in G$.
If $V$ is  orthogonal with $q$ even 
then $|x|=r^n-1$. In the unitary case $\det (x)\in \FF_{q}$
implies $|x|=(r^n-1)/(q+1)$. Set $T=\lan X\ran$.

\begin{lem}\label{nm11}
The subgroup $T$ is self-centralizing.
\end{lem}

\begin{proof} 
Let $\rho_1,\rho_2:T\ra GL_n(q)$ be the representations defined by
$\rho_1(X)=x$ and $\rho_2(X)=\sigma ( x^t)^{-1}$. Then
$\rho_1,\rho_2$ are non-equivalent. 
Indeed, $\sigma ( x^t)^{-1}=gxg\up$ for some $g\in GL_n(r)$ \ii
$\sigma  (x)^{-1}=hxh\up$ for some $h\in GL_n(r)$ (as $x$ and
 $x^t$ are similar matrices). Hence the claim follows from Lemma \ref{8h10}.

By Schur's lemma,  $C_{G(V)}(X)=\diag (C,\si(C\up))$,  where $C$
runs over $C_{GL_n(r)}(x)$, which is isomorphic to
$F_{r^n}^\times$. So $C_{G(V)}(X)$ is cyclic. \end{proof}

\med
With respect to the above basis
  matrices of $A$
  are of shape:

\begin{equation}\label{e1}
\begin{pmatrix}
 \Id_n&   b  \cr
 0&\Id_n \end{pmatrix},~~
\end{equation}
where $b\in GL_n(r)$.
We observe that $H$ normalizes $A$ and $C_H(A)=Z(G(V))$.

\noindent Obviously, $A$ is an abelian group of exponent $p$.  Let
$L,L_0$ be as in (1) and (2) in Section 2 above.

\begin{lem}\label{222}
Let $a\in A$ be
a matrix of shape $(\ref{e1})$.

{\rm (i)} Suppose that $V$ is not an orthogonal space in
characteristic $2$. Then
$a\in G(V)$ \ii $b\in L$.

{\rm (ii)} Suppose that $V$ is  an orthogonal space in characteristic
$2$. Then
$a\in G(V)$ \ii $b\in L_0$.

{\rm (iii)} The mapping $M\ra b$ is a bijection $A\ra L$, unless $p=
2$ and $\si= 1$ when this is a bijection between $A$ and $L_0$.

{\rm (iv)} If $V$ is orthogonal then $A < \Omega^+_{2n}(q)$.
\end{lem}

\begin{proof}
For (i)--(iii) see \cite[p.240]{DZ}. For (iv), suppose first
that $p\neq 2 $. Set $\Omega=\Omega_{2n}(q)$. As $\Omega$ is a
normal subgroup in $G(V)$ of index 4,   the result follows.
Suppose that $p=2$.
  In this case $|G(V):\Omega|=2$, so either $HA< \Omega$ or
$HA$ has a  subgroup of index 2. As $p=2$ and $n>2$, $H$ has no
subgroup of index 2. So $A< \Omega$ unless $A$ has an
$H$-stable subgroup of index 2. This contradicts \cite[Lemma
4.6]{EZ}, where it is shown that $L_0$ is an \ir $\FF_rH$-module
(provided that $n>2$).
\end{proof}

\medskip
Note that $A < SU_n(q)$ in the unitary case.

\begin{lem}\label{333}
Every $P/Z(G)$-orbit on  $A\setminus \{1\}$
contains a regular $T/Z(G)$-orbit.
\end{lem}

\begin{proof}
This follows from Lemma  \ref{p44}. Indeed, by Lemma
\ref{222}(i), resp., \ref{222}(ii), $A$ and $L^+$, resp., $L_0^+$,
are  isomorphic $H$-sets. (Here $L^+$, resp. $L_0^+$, is the
additive group of $L$, resp. $L_0$.) The action of $H$ on $A$ by
conjugation corresponds to the congruence action of $GL_n(r)$ on
$L$, resp. $L_0$,  and hence we may use Lemma \ref{p44}.
\end{proof}

\medskip
Recall that $D_\si=\{g\in Z(GL_n(r)): g \si(g)=1\}$, see Lemma
\ref{dd4}. Let $D$ be the image of $D_\si $ in $H$ under the
isomorphism $GL_n(r)\ra H$. It is clear that  $D  \leq
Z(G(V))$.

\begin{prop}\label{872}
Let 
$G/Z(G)=P\Omega^+ _{2n}(q)$ or $PSU_{2n}(q)$. Let  $M$ be a
non-trivial \ir $\FF G$-module with central character $\zeta$.
 Then $M|_{T}$ contains a submodule isomorphic to ${\rm Ind}_{Z(G)}^T\zeta $.
In particular, if $\zeta=1_{Z(G)}$ then $M|_{T}$ contains  a
regular submodule, and hence $1_T\in M|_T$.\end{prop}

\begin{proof} This follows from Lemmas \ref{ab2} and \ref{333}.
Indeed, as $(p,|Z(G)|)=1$, there is a $p$-subgroup $A_1$ in $G$
which projects to $A$ under the \ho $G\ra \Omega_{2n}(q)$, and the
preimage of $A$ coincides with $A_1Z(G)$. If $T_1$, resp., $S$, is
the preimage of $T$, resp. $SL_n(q)$, in $G$ then $A_1$ is $T_1$-
and $S$-invariant. Moreover, $Z(G)$ is in the kernel of the
conjugation  action of $T_1$ and $S$ on $A_1$, so $A_1$ can be
viewed as $T_1/Z(G)$ and $S/Z(G)$-sets. Then $A$ and $A_1$ are
isomorphic with respect to these actions.

We can decompose $M$ as a direct sum of homogeneous $A_1$-modules
$M=\oplus_{\al\in \Irr A_1} M_\al$. Obviously, this decomposition
can be arranged as follows: $M=\oplus_{O}(\oplus_{\al\in O}
M_\al)$, where the $O$'s are $N_G(A_1)$-orbits on $\Irr (A_1)$. As
$[S,A_1]\neq 1$, it follows that there exists $\al\in\Irr (A_1)$,
$\al\neq 1_{A_1}$ such that $S\al\neq \al $. Note that the
conjugation action of $S$ on $A_1$ is realized via the congruence
action of $SL_n(r)$ on $L$ or $L_0$ (see Section 2). Fix this
$\al$, and let $O_1=S\al$ be the $S$-orbit of $\al$. By Lemma
\ref{p44},  there is $\beta\in O_1$ such that $t\beta=\beta$ for
$t\in T$ implies $t\in D$. As  $D  \leq Z(G(V))$, it follows that
$t\beta=\beta$ for $t\in T_1$ implies $t\in Z(G)$. Thus,
$C_{T_1}(\beta)=Z(G)$. Let $\beta|_{Z(G)}=\zeta$. Then, by Lemma
\ref{ab2}, $M|_{T}$ contains a submodule isomorphic to ${\rm
Ind}_{Z(G)}^T(\zeta)$. \end{proof}

\subsection{The groups  $G$ with $G/Z(G)=\Omega_{2n+1}(q)$, $n>2$,
$q$ odd, or  $PSU_{2n+1}(q)$, $n\ge 1$}

Let $V$ be  a unitary  or orthogonal space of dimension $2n+1$
over $\FF_r,$ where $r=q^2$  in the unitary case, and $r=q$
otherwise. Let $W$ be a maximal totally isotropic subspace of $V$,
so $\dim W=n$. Let $G_1=\Omega_{2n+1}(q)$ in the orthogonal case,
and $G_1=SU_{2n+1}(q)$ in the unitary case.

We wish to mimic the reasoning in Subsection 3.3; in particular,
the groups $A,T$ here are analogous to those above. However, Lemma
\ref{333} is not true for our current situation, as $C_T(A)$ is
not necessarily in $Z(G)$.

Let $B=\{b_1\ld b_{2n+1}\}$ be a basis of $V$ such that $b_1\ld
b_n\in W$
 and the Gram matrix $\Gamma$ corresponding to $B$ is of shape

$$ \begin{pmatrix}0&0&\Id_n\cr 0&1&0\cr \Id_n&0&0\end{pmatrix} $$
Then $V=W\oplus V_0\oplus W_1$, where $V_0=\lan b_{n+1}\ran$ and
$W_1=\lan b_{n+2}\ld b_{2n+1}\ran$. So $V_0^\perp =W+W_1$. By
definition, $G(V) = \{ g \in GL_{2n+1}(r): g\Gamma \sigma (g)^t =
\Gamma\}$. Let $H$ be the stabilizer in $G_1$ of each of the
subspaces $W,W_1,V_0$
of this decomposition. 
Then $H$ consists of matrices

\begin{equation}\label{eq5}
\begin{pmatrix}h&0&0\cr 0&f&0\cr 0&0&\si(h^t)\up\end{pmatrix},
\end{equation}
where $h\in GL_n(r)$ and  $f\in G(V_0)$. If $V$ is orthogonal then
$H< SO_{2n+1}(q)$, whence  $f=1$. Note that $V_0^\perp$ has Witt
index $n$. Therefore, $O(V_0^\perp)=O^+_{2n}(r)$, and hence $H$ is
isomorphic to a subgroup of index 2 in $GL_n(r).$
 If $V$ is unitary then
 $f\cdot \det h\cdot \det \si(h\up)=1$.
\itf $H$ is isomorphic to $GL_n(r)$.

Let $x\in GL(W)$ be a Singer cycle, so $|x|=r^n-1.$ We set

$$X=\begin{pmatrix}x&0&0\cr 0&e&0\cr 0&0&\sigma ( x^t)^{-1}
\end{pmatrix}~~~{\rm and}~~~T=\lan X\ran,$$
where $e= \si(\det x)\cdot\det x\up$. So $T$ is a cyclic subgroup
of $G(V)$. Note that $ x^{(r^n-1)/(r-1)} $ generates  the center
of $GL_n(r)$ and $C_{GL_n(r)}(x)=\lan x \ran$.

\begin{lem}\label{na1}
Let 
 $T=\lan X \ran$.
Then one of the \f holds:

\medskip
{\rm (i)} $C_{G_1}(X)=T$.

\medskip
{\rm (ii)} $V$ is unitary of dimension $3$ and $q=2$ or $3$.
\end{lem}

\begin{proof}

Suppose that (ii) does not hold. 
It follows from Lemma \ref{8h10} that the \reps $h\ra h$ and $h\ra
\si(h^t)\up $ are not equivalent.  By Schur's lemma,
$C_{G(V)}(X)=\diag (A,b,C)$, where $C= \sigma ( A^t)^{-1}$,  $A$
runs over $C_{GL_n(r)}(x)=\lan x\ran$ and $b\in G(V_0)$.  Then
$C_{G(V)}(X)$ is obviously abelian. If $V$ is orthogonal then
$\det A\cdot \det C=1$, and hence $b=1$. So (i) follows. Suppose
that $V$ is unitary. Then $b \cdot\det A\cdot \si(\det C)=1$.   As
$A=x^k$ for some $k$ by Schur's lemma, we have $C=\sigma (
x^t)^{-k}$, and hence $b=\det x^{-k} \sigma ( x)^{k}=e^k$, as
required. 
\end{proof}

With respect to the above basis  consider the subgroups  $A,U,D$
of $G$, consisting of matrices, respectively,  of shape

\med

\begin{equation}\label{e2}\begin{pmatrix}
 \Id_n&0&   \al\cr
0&  1&0\cr 0&0&\Id_n\end{pmatrix}, ~~~\begin{pmatrix}
 \Id_n&\beta& \gamma   \cr
0&  1&-\sigma (\beta)^t\cr 0&0&\Id_n\end{pmatrix}, \begin{pmatrix}
 \lam \cdot \Id_n&0&   0\cr
0&  \lam^{-2n}&0\cr 0&0&\lam \cdot
\Id_n\end{pmatrix},
\end{equation}
where $\lam\in F_{r}$ and
$\si(\lam)\lam=1$. In particular, if $V$ is an orthogonal space
then $\lam^2=1$, and hence $|D|\leq 2$ and $D< T$. (It
follows from \cite[2.7.2 and 2.5.13]{KL90}   that $|D|=1$ \ii
$(q-1)n/2$ is even.)

Suppose that $V$ is unitary. Then $\lam\in U_1(q)$.
We write $a=a(\al)$
for $a\in A$ and $u=u(\beta,\gamma)$ for $u\in U$. Simple
computation shows that $u=u(\beta,\gamma)\in U$ is equivalent to
$\gamma + \sigma (\gamma^t) + \beta^t\sigma (\beta) = 0$. In
particular, $a(\al)\in A$ is equivalent to $\al + \sigma
(\al^t)=0$.
Note that 
$D=\lan Y\ran $, where $Y=X^{(r^n-1)/(q+1)}$. The  matrices in $D$
are of shape $ \lam\diag(\Id,\lam^{-2n-1},\Id)$, so $|D|=q+1$. In
addition,  $D$ contains $ Z(G_1)$.

It is obvious that $H$ normalizes $A$. In addition, $A=Z(U)=U'$,
where  $U'$ is the derived group of $U$ (see \cite[Lemma 3.1]{DZ}).

\begin{lem}\label{ga5} $D=C_H(A)$ and $D=C_T(A)$. In addition,
$C_D(u)=Z(G_1)$ for every $u\in (U\setminus A).$
\end{lem}

\begin{proof}
The first statement means that $D$ is the kernel of the
conjugation action of $H$ on $A$. As explained in Section 2, $H$
acts on $A$ via the action of $GL_n(r)$ on $L$, so the claim
follows from Lemma \ref{dd4}. The second statement follows as
$D< T$. The additional assertion is clear from the shape of
matrices of $D$ and $U.$
\end{proof}

\medskip
Define $H_0$ to be the subgroup of $H$ whose matrices are
$\diag(h,1,\si(h^t)\up)$ with $h\in SL_n(r)$.

\begin{lem}\label{g2}
{\rm (i)} Suppose that $u\in ( U\setminus A)$ and $g\in D $ is
non-scalar. Then   $[g,u]\notin A$. (This means that the
$D$-orbits on $U\setminus A$ are isomorphic to each other and have
size $|D/Z(G_1)|$.)

{\rm (ii)} If $1\neq a\in A$, then there exists $h\in H_0$ such that
$C_T(hah\up)=D$.

{\rm (iii)} Let $\chi$ be an \ir character of $A$, $\chi\neq 1_A$.
Then there exists $g\in H_0$ such that $gTg\up \cap C_H( \chi)=D$.
\end{lem}

\begin{proof}
(i) is obvious from the shape of the matrices above.

 (ii). Suppose
first that $n=1$. Then $C_{H}(a)$ consists of matrix of shape
satisfying $h\al \si(h)=\al $, or $h \si(h)=1$, that is, $h\in
U_1(q)$. \itf $C_{H}(a)=D$, that is, (ii) holds for $g=1$.

Suppose  $n>1$. Then $C_{H}(a)$ consists of matrix of shape
(\ref{eq5}) satisfying $h\al \si(h^t)=\al $, or $h\al =\al
\si(h^t)\up$. As $h\ra \si(h^t)\up$ is an \irr of $GL_n(q)$, by
Schur's lemma $\al$ is invertible. As $\al + \sigma (\al)=0$, the
matrix $\al$ can be regarded as a matrix of a non-degenerate
unitary form on $W$, and the condition $h\al \si(h^t)=\al $
determines a unitary group preserving the form. Now
$U_n(q)$ is a proper subgroup of  $GL_n(q^2)$, and
$SU_n(q)\neq SL_n(q^2) $.

This conclusion is necessary in order to use Lemma \ref{sz1}.
Consider  the $H_0$-orbit $\Omega=\{hah\up: h\in H_0\} $, and
$Z_1=\ker \Omega :=\{h\in H_0: h\om=\om$ for all $\in \Omega \}$.

Recall that every normal subgroup of $GL_n(r)$ that does not
contain $SL_n(r)$ belongs to the center of $GL_n(r)$, unless
$(n,r)=(2,2),(2,3)$. As here $r$ is a square in the unitary case
and $n>2$ in the orthogonal case, these exceptions do not occur.
So $Z_1\leq  Z(H_0)$, and hence $Z_1$ consists of matrices of
the shape $\diag(s\cdot \Id,s\up \si(s) ,\si(s)\up \cdot \Id )$,
where $0\neq s\in \FF_r$. As in the case $n=1$ above, one observes
that $Z_1=D$. Thus, $H_0/D$ acts faithfully on $\Omega$. Let
$\Omega_1$ denote the set of
$Z_1$-orbits on $\Omega$. As 
$\Omega$ is a transitive $H_0$-set, it follows that all
$Z(H_0)$-orbits are isomorphic to each other, and have $Z_1=D$ in
the kernel, whence each is the regular $Z(H_0)/D$-set.
In addition, $H_0$ permutes these transitively. Applying Lemma
\ref{sz1} to $\Omega_1$, we find some $\om_1\in \Omega_1$ such
that the orbit $(T/Z(H_0))\cdot \om_1$ is regular. This means that
$t\om_1=\om_1$ for $t\in T$ implies $t\in Z(H_0)$. Next pick any
$\om\in\om_1$. As $Z(H_0)/Z_1$ acts regularly on $\om_1$, it
follows that $t\om=\om$ implies $t\in Z_1$. As $Z_1$ is in the
kernel of $\Omega$, the orbit $T\om$ is a regular $T/Z_1$-orbit.
As $\om =gag\up$ for some $g\in H_0$, the result follows.

(iii) follows from (ii) as $\Irr(A)$ and $A$ are isomorphic
permutation sets for Aut$\,A$.
\end{proof}

\begin{remar}\label{rr1}
{\em If $V$ is orthogonal and  $n(q-1)/2$ is even then $D=1$
as mentioned above. If $V$ is unitary then $D=Z(G)$ \ii
$|Z(G)|=q+1$, and \ii $q+1$ divides $2n+1$. Therefore, in this
case Proposition \ref{ab1} follows from Lemma \ref{g2}.}
  \end{remar}

\begin{prop}\label{872a}
Let 
$G/Z(G)= P\Omega _{2n+1}(q)$. Let 
$M$ be a non-trivial \ir
$\FF G$-module with central character $\zeta$. 
 Then $M|_{T}$ contains a submodule isomorphic to ${\rm Ind}_{Z(G)}^T\zeta $.
In particular, if $\zeta=1_{Z(G)}$ then $M|_{T}$ contains  a
regular submodule, and hence $1_T\in M|_T$.\end{prop}

\begin{proof} By the above, $T$ acts faithfully on $V_0^\perp=W+W_1$. Let
$0\neq v\in V_0$, and let $K=C_G(v)$. Then $K\cong
\Omega(V_0^\perp)$. As the Witt index of $V_0^\perp$ equals $n$,
it follows that $O(V_0^\perp)\cong O^+_{2n}(q)$. Obviously,
$[D,K]=1$. If $D=1$ then the result follows from Lemma \ref{872}.

Let $|D|=2$. Then  $M=M_1\oplus M_2$, where $D$ acts trivially on
$M_1$. As $D$ is not scalar in $G$, both $M_1,M_2$ are non-zero
and $K$-stable. 

We show that $K|_{M_i}$ is non-trivial for $i=1,2$. Let $M'$
denote the subspace spanned by all $(a-\Id)M$ for $a\in A$. This
space is stable under $C_G(A)$, in particular, under $U=O_p(P)$
and $D$. It suffices to show that $D$ is not scalar in $M'$.
Suppose the contrary. Let $\tau (g)=g|_{M'}$ for $g\in C_G(A)$.
Obviously, $\ker\tau\cap A=1$. \itf $\ker\tau\cap U=1$ (otherwise,
$[U,\ker\tau]\leq A\cap \ker\tau=1$, and hence
$\ker\tau\leq  Z(U),$ contrary to Lemma \ref{g2}(1)). Then
$\tau ([d,u])\neq 1$ for $u\in (U\setminus A)$, $1\neq d\in D$ by
\ref{g2}(3). Therefore, $\tau (d)$ is not scalar, and hence $d$
has \eis 1 and $-1$ on $M'$.

To complete the proof let $M_i=\sum_O\sum_{\al\in\Irr
A}M_i^{\al}$, where $M_i^{\al}=\{m\in M_i: am=\al(a)m $ for all
$a\in A\}$. Then use Lemma \ref{872}.\end{proof}

\begin{lem}\label{dh2} Let $X=Y E$, where $E$ is an
extraspecial $p$-group normal in $X$, and  $Y=\lan y\ran$ is a
cyclic group and let $Z=C_Y(E)$. Suppose that $[Y,Z(E)]=1$ and
$C_Y(e)=Z$ for every $e\in E\setminus Z(E)$. Let $M$ be a faithful
\ir $\FF X$-module non-trivial on $[E,E]$. Let $M|_{Z}=\lam \cdot
\Id$. Then either $M=M_1+N$ or $M_1=M+N$, where $\dim N=1$,
$M_1=m\cdot {\rm Ind}_{Z}^Y \lam$, and $m+\delta=\dim M/|Y/Z|$ for
some $\delta\in  \{1,-1\}$.
\end{lem}

\begin{proof}
This is a small refinement of \cite[Theorem 9.18]{DH}.
Indeed, the lemma coincides with  \cite[Theorem 9.18]{DH} if
$Z=1$. In general, let $k=|Y/Z|$ and let $s$ be a scalar matrix of
order $|Y|$ such that $s^k=y^k$.
Set $S:=\lan s\up y\ran$ and $X_1=
 \lan E,S\ran$. Then $X_1$ is a semidirect product of
$E$ and $S$, and $C_S(E)=1$. In addition, all non-trivial orbits
of $S$ on $E/Z(E)$ are of the same size as $S$.
This means that $SE$ satisfies
the assumptions of \cite[Theorem 9.18]{DH}, saying that in this
case $M_1$ is a free $\FF S$-module of rank $m$ with $m$ as above.
As the element $y$ is a scalar multiple of $s\up y$, our
conclusion on $M|_Y$ follows from the result about $M|_S$.
\end{proof}

\begin{lem}\label{bg3}
 Let $\tau$ be an \irr of the group $DU$ nontrivial on $U'$.
 Then every \irr $\lam$ of $D$ such that $\lam |_{Z(G_1)}=\zeta\cdot\Id$
is a constituent of $\tau|_D$, unless $Z(G_1)=1$ and
$\dim\tau=q$.
\end{lem}

\begin{proof} Recall that $C_D(U)= Z(G_1)$, see Lemma \ref{ga5}.
Let $U_1=\{u\in U:\tau (u)$ is scalar$\}$. Then $U\neq U_1$. Set
$E=\tau(U)$, and $x=|E/Z(E)|$. By \cite[Lemma 3.13]{DZ},
$E=Z(E)\cdot E_1$, where $E_1$ is an extraspecial group and $x$ is
a $q$-power. As $\tau$ is irreducible, $\tau(U_1)=Z(E)$. Let
$x=r^k=q^{2k}$. Obviously, $|E_1/Z(E_1)|=x$. It is well known that
an \irr of $E_1$ is either one-dimensional or of degree
$\sqrt{x}$, in our case this is $q^k$. So $\dim\tau=q^k$ for some
$k$ (this is also stated in \cite[Corollary 12.6]{GMST}). In
addition, $E_1$ can be chosen $D$-stable. Indeed, $U/A$ is an
$F_pD$-module and $U_1/A$ is obviously a submodule. By Maschke's
theorem, there is a $D$-stable complement $U_2/A$. Then
$E_1=\tau(U_2)$ is $D$-stable.

 Note that $U_2/A\cong E_1/Z(E_1)$, and this is an
$F_pD$-module isomorphism. By
 Lemma \ref{g2}(1), the $D$-orbits on
 $U/A$ are of size $|D/Z(G_1)|$, and hence so are  the $D$-orbits
 on $E_1/Z(E_1)$. As $D$ acts trivially on $A$
and $\tau(A)$, we may apply Lemma \ref{dh2}  in order to claim
that $\tau|_D$ contains a submodule isomorphic to ${\rm
Ind}_{Z(G_1)}^D(\zeta)$, unless $\dim \tau+1=|D/Z(G_1)|$. As
$|D|=q+1$ (see comments prior to Lemma \ref{ga5}), the lemma
follows.
\end{proof}

\begin{prop}\label{8.5}
Let $G = SU_{2n+1}(q)$, where $(2n+1,q+1)\neq 1$.  Let  $T <  G$
be as in Lemma $\ref{na1}$. Let $\phi$ be a non-trivial
irreducible $\FF $-\rep of $G$ with central character $\zeta$.
Then the restriction $\phi |_T$ contains all \ir representations
of $\tau\in\Irr(T)$ such that  $\tau|_{Z(G)}=\zeta\cdot \Id$.
\end{prop}

\begin{proof} Now $G=G_1$.  Let $H,A$ be as above.
Recall that $H$ acts on $A$ by conjugation, and this action
translates in the usual way to an action of $H$ on $\Irr (A)$ by
setting $\al^n (a) = \al (nan^{-1})$  for $n \in H$, $a\in   A$
and $\al\in \Irr (A)$. Let $M$ be an $\FF G$-module afforded by
$\phi$. For $\al_i\in \Irr (A)$ set $M_{\al}  = \{ m \in  M | g m
= \al(g) m ~{\rm~ for~ all} ~ g \in A\}.$ We can write
$$M|_A=\oplus_{O}\oplus_{\al\in O}M_{\al},$$ where $O$ runs over
the orbits of $H$ in $\Irr (A)$. For any orbit $O$, the subspace
$M_O = \sum _{\al\in O}  M_\al $ is an $H$-submodule of $ M|_H$.
Since   $A\cap Z(G)=1$, $M|_A$ is a faithful $A$-module.
Therefore, there exists an orbit $O$ such that $A$ acts
non-trivially in $ M_O$, and we fix this $O$ from now on. Then $A$
acts non-trivially on every  $M_{\al}$ for $\al\in O$. It is easy
to observe that $A=U'=Z(U)$, see, for instance, \cite[Lemma
3.1]{DZ}. In addition, $[D,A]=1$ (Lemma \ref{ga5}). Therefore,
every $M_{\al}$ is $DU$-stable for $\al\in O$.

By Lemma \ref{g2}(2),   the group  $T/D$ has a regular orbit in
$O$.  Let
$\al_0\in O$ be such that the $T$-orbit $O'$ 
is of size $ | T/D |$. Set $M' :=\oplus _{\al\in O'} M_{\al}.$

Observe that $D$ acts trivially on  $\Irr (A)$. Therefore, the
action of $T$ on $M_O$ is realized via  the action of $D$ inside
each $M_{\al}$   and the action of $T/D$ on the $M_{\al}$'s for
$\al\in O'$,  which are regularly permuted by $T/D$. Let $N$ be an
\ir submodule of $M'|_{DU}$. Then $A=U'$ is non-trivial on $N$. As
shown in the proof of Lemma \ref{bg3}, $\Ind_{Z(G)}^D(\zeta)$ is a
submodule of $N|_D$. Then $\Ind_{Z(G)}^D(\zeta)$ is a submodule of
$M'|_{T}$, and the result follows. (We also use $\zeta$ to denote the module
afforded by the character $\zeta$.)
\end{proof}

\medskip
{\bf Proof of Proposition} \ref{ab1}. For the groups $SL(n,q)$ and
$P\Omega_{2n}^-(q)$, $n>3$, the result is contained in Lemma
\ref{ca1}, whereas for $PSU_{2n}(q)$, $n>1$, and
$P\Omega_{2n}^+(q)$, $n>3$, this follows from Lemmas \ref{nm11}
and \ref{333}. See  Remark \ref{rr1} for $PSO_{2n+1}(q)$ with
$n(q-1)/2$ odd.

\subsection{The exceptional case}

In this subsection we assume that $G=SU_n(q)$, where $n \geq 3$ is
odd and coprime to $q+1$. This implies that $Z(G)=1$. As above, we
set $r=q^2$. Let $P$ be the stabilizer of an isotropic line $W$ at
the natural $G$-module $V$. Choose any complement $V_1$ to $W$ in
$W^\perp$. Then there exists a basis $b_1\ld b_{n}$ of $V$ whose
Gram matrix is $$\begin{pmatrix}0&0&1\cr 0&\Id_{n-2}&0\cr 1&0&0
\end{pmatrix}.$$
Then  $b_1\in W$ and $b_2\ld b_{n-1}\in W^\perp$. Set $V_1=\lan
b_2\ld b_{n-1}\ran$. Let $H,U$ be the subgroups of $G$ consisting
of matrices of shape $$\begin{pmatrix}\si (a\up)&0&0\cr 0&y&0\cr
0&0&a
\end{pmatrix}, ~~~~~\begin{pmatrix}1&-\si(v^t)&w\cr 0&\Id_{n-2}&v\cr
0&0&1
\end{pmatrix},$$ where $0\neq a\in \FF_r$, $y\in U(V_1)$ and $v\in V_1$.
Here $v$ is an arbitrary element of $\FF_r^{n-2}$. The entry $w$
depends on $v$ but we do not need to express the dependence
explicitly; the determinant condition is $a\si(a\up)\det y=1$.

Note that $Z(U)$ consists of all matrices in $U$
with $v=0$. The conjugation action of $H$ on $U$ induces on
$U/Z(U)$ the structure of $\FF_rH$-module. Moreover, the subgroup
$\{\diag(1,h,1):h\in SU(n-2,q)\}$ acts on $U/Z(U)$ exactly as on
$V_1$.

It is well known that $U(V_1)\cong U_{n-2}(q)$ contains a
self-centralizing cyclic subgroup $\lan t\ran$ of order
$q^{n-2}+1$. Set $T=\lan X\ran$, where 
$X=\diag(a,t,\si(a\up))$ and  $a$  is a generator of
$\FF_{r}^\times$ such that $\det X=1$. As $\si(a)=a^{q}, $ we have
 $\det t=a^{q-1}$.

\begin{lem}\label{cc1} {\rm (i)} $T=C_G(X)=C_{H}(X) $ is a cyclic 
$p'$-group.

\med
{\rm (ii)} If $1\neq c\in T  $, $u\in (U\setminus Z(U))$ then
$[c,u]\notin Z(U)$. In other words, every non-identity element
$c\in C $ acts on $U/Z(U)$  fixed point freely.
\end{lem}

\begin{proof} The matrix  $X$ gives rise to three \ir \reps of $T$,
namely, $X\ra t$, $X\ra a$ and $X\ra a^{-q}$. As $a^{q+1}\neq 1$,
these are pairwise non-equivalent. (If $n=3$ then $t\in \FF_r$
 and $t^{q+1}=1$; so the claim is true for $n=3$ as well.)
 Then, by
Schur's lemma, the elements of $ C_{G(V)}(X)$ are of shape $\diag
(b^{-q},t^i,b)$ for some $i$ and $b\in \FF^\times_r$. So
$(|T|,p)=1$.

\med
(ii) Suppose the contrary. Then there is $0\neq v\in V_1$ such
that $t^ib\up v=v$. Hence $v$ is a $b$-eigenvector for $t^i$.
Therefore $t$ stabilizes the  $b$-eigenspace of $t^i$ on $V_1$. As
$t$ is \ir on $V_1$, it follows that the $b$-eigenspace  coincides
with $V_1$, and hence $t^i=b\cdot \Id$. Then $t^i\in Z(U(V_1))$,
whence $b^{q+1}=1$ and $b^{-q}=b$. This means that $c$ is scalar,
a contradiction.

 (i) If $C$ is not cyclic then $C$ contains a non-cyclic
elementary abelian $s$-subgroup of order $s^2$ for some prime $s$.
It is easy to observe that this contradicts (ii).
\end{proof}

\begin{prop}\label{357} Let $G=SU_n(q)$, where
$n \geq 3$ is odd  and  coprime to $q+1,$ and let $T$ be as above.
Let $\phi$ be an \irr of  $G$. Then $\phi_{T}$ contains every \irr
of $T$ as a constituent, unless, possibly, $\phi$ is a Weil \rep
of $G$.
\end{prop}

\begin{proof} Let $P$ be the stabilizer of a one-dimensional isotropic
subspace at the natural  $G$-module. Then $P$ is a parabolic
subgroup of  $G$. Let $U$ be the unipotent radical of $P$. By
\cite[Corollary 12.4]{GMST}, the character $\phi|_U$ contains a
non-trivial linear character $\chi$, say, of $U$, unless
$\phi=1_G$ or a Weil \rep of $G$.

Now $\chi |_{Z(U)}$ is trivial, since $Z(U)=U'$, and hence $\chi$
can be viewed as a character of $U/Z(U) $. Denote the character
group of $U/Z(U)$ by $\Omega$, so $\chi\in\Omega$. The action of
$T$ on $\Omega$  is dual to that on $U/Z(U)$. By Maschke's
theorem, $U/Z(U)$ is a completely reducible $\FF_rT$-module. It
follows from Lemma \ref{cc1}(2) that every element of  $T$ acts
fixed point freely on $\Omega\setminus\{1\}$, in particular, the
$T$-orbit of $\chi$ is of length $|T|$.

Let $M$ be the module afforded by the \rep $\phi$, and let $M'$ be
the subset of fixed vectors for  $Z(U)=U'$. Then $\chi$ is the
character of a constituent  of the restriction $M'|_U$, and
obviously, $TM'=M'$. Therefore,
 $M'=\oplus _{\om\in\Omega}M_\om$, where $M_\om :=\{m\in M':
 um'=\om (u)m'\}$.
 As $M_\chi\neq 0$, we get a non-trivial $T$-module
 $N:=\oplus _{c\in T}M_{c\chi}$, where $\chi\ra c\chi$ ($c\in T$) is the  action of $T$
 on $\Omega$ defined above. Moreover, $N|_T$ is  regular 
 as the $T$-orbit of $\chi$ is regular and $Z(G)=1$ (see Lemma  \ref{ab2}). It follows
 that $N|_T$
 contains every \ir $T$-module. 
\end{proof}

Next we establish the converse of Proposition \ref{357} in the case
of complex representations.

Recall  that {\it the generic  Weil character} $\om_{n,q}$ of
$U_n(q)$ is defined by $$\om_{n,q}~:~h \mapsto
(-1)^n(-q)^{\dim_{\FF_{q^2}} \Ker(h-1)},$$ where the dimension in
the exponent is on the natural module for $G$ over $\FF_{q^2}$.
For $n>2$ this is the sum $\sum^{q}_{l=0}\om^{l}_{n,q}$ of $q+1$
irreducible characters $\om^l_{n,q}$, $0\leq l\leq q$ (we refer to
them as \ir Weil characters).

\begin{thm}\label{complex-su}
Let $G = SU_n(q)$, where $n \geq 3$ is coprime to $2(q+1)$, and
let $T$ be the maximal torus of $G$ of order $(q-1)(q^{n-2}+1)$
described above.

{\rm (i)} If $\chi $ is any nontrivial non-Weil \ir character of
$G$, then $\chi|_T$ contains every irreducible character of $T$.

{\rm (ii)} If $\chi $ is any of the $q+1$ Weil \ir characters of
$G$, then $\chi|_T$ contains all but $(q-1)$ irreducible
characters of $T$.

{\rm (iii)} If $\mu \in \Irr (T)$, then $\Ind^{G}_T(\mu)$ contains
all but possibly one nontrivial irreducible character of $G$, and
the missing character must be a Weil character of $G$. The total
number of $\mu \in \Irr (T)$ such that $\Ind^G_T(\mu)$ misses a
nontrivial irreducible character of $G$ is $q^2-1$.

{\rm (iv)} The only irreducible character of $G$ which is not a
constituent of $\Ind^G_T(1_T)$ is the (unipotent) Weil character of
$G$, of degree $(q^n-q)/(q+1)$.
\end{thm}

\begin{proof}
1) Note that (i) is already proved in Proposition \ref{357}. For
brevity, set $q_1 := q-1$ and $q_2 := (q^{n-2}+1)/(q+1)$. Recall
that $T = G \cap T_1$, where $T_1$ is a maximal torus of $U_n(q)$.
Let $V$ be the natural module for $U_n(q)$ and $\overline{V}:=V
\otimes_{\FF_{q^2}} \overline{\FF}_q$. One observes that $T_1$ can
be diagonalized under a suitable basis of
$
\overline{V}$ as follows: $$T_1 = \left\{ g_{i,j} :=
\diag(a^{i},a^{-qi},b^{j},b^{-qj}, \ldots ,b^{q^{n-3}j})
  \mid 0 \leq i \leq q^2-2,~0 \leq j \leq q^{n-2} \right\}.$$
 Here, $a,b \in
\overline{\FF}_q^{\times}$ are some fixed elements of order
$q^2-1$, respectively $q^{n-2}+1$, chosen such that $a^{q_1} =
b^{q_2} =: c$. We also fix $\al,\beta \in \CC^{\times}$ of order
$q^2-1$, respectively $q^{n-2}+1$, such that $\al^{q_1} =
\be^{q_2} =: \ga$. Then every irreducible character of $T_1$ is of
the form $\lam = \lam_{s,t}~:~g_{i,j} \mapsto \al^{si}\be^{tj}$
with $0 \leq s \leq q^2-2$, $0 \leq t \leq q^{n-2}$.

\medskip
2) The Weil \ir characters are of shape
 $$\om^k_{n,q}~:~h \mapsto
\frac{(-1)^n}{q+1}\sum^{q}_{l=0}\ga^{kl}
  (-q)^{\dim_{\FF_{q^2}} \Ker(h-c^l)},$$
cf. \cite{TZ96}. (Here, $\Ker(h-c^l)$ is computed on the natural
module $V$ for $G$.)
Observe that  $$\dim_{\FF_{q^2}}
\Ker(g_{i,j}-c^l) = \left\{
\begin{array}{rl}
  0, & q_1 \not{|} i,~ q_2 \not{|} j,\\
  2\delta_{i_1,l}, & i = q_1i_1,~q_2 \not{|} j,\\
  (n-2)\delta_{j_2,l}, & q_1 \not{|} i,~j = q_2j_2,\\
  2\delta_{i_1,l}+(n-2)\delta_{j_2,l}, & i = q_1i_1,~j = q_2j_2,
 \end{array} \right.  $$
where $i_1,j_2 \in \ZZ$ and $0 \leq i_1,j_2 \leq q$. It follows
that $$[\om^{k}_{n,q}|_{T_1},\lam_{s,t}]_{T_1} = $$
$$-\frac{1}{(q+1)|T_1|}\sum^{q}_{l=0}\ga^{kl} \cdot
    \left(\sum^{q^2-2}_{i=0}\al^{-si} + (q^2-1)\ga^{-sl}\right) \cdot
    \left(\sum^{q^{n-2}}_{j=0}\al^{-tj} - (q^{n-2}+1)\ga^{-tl}\right) = $$
$$-\frac{1}{q+1}\sum^{q}_{l=0}\ga^{kl} \cdot (\delta_{0,s} + \ga^{-sl})
  \cdot (\delta_{0,t} - \ga^{-tl}) =
\delta_{0,s}\delta_{0,k-\bar{t}} + \delta_{0,k-\overline{s+t}}
  - \delta_{0,s}\delta_{0,t}\delta_{0,k} - \delta_{0,t}\delta_{0,k-\bar{s}},$$
where for each $i \in \ZZ$ we choose $0 \leq \bar{i} \leq q$ such
that $q+1$ divides $(i-\bar{i})$. Thus $\lam_{s,t}$ is a
constituent of $\om^{k}_{n,q}|_{T_1}$ precisely when $k =
\overline{s+t}$ and $t \neq 0$, in which case this multiplicity is
$1$ if $s \neq 0$ and $2$ if $s = 0$.

\smallskip
3) So far we have used only the assumption that $n$ is odd. Now we
take into account the hypothesis $(n,q+1) = 1$, which implies that
$T_1 = T \times Z$ for $Z := Z(U_n(q))$. Clearly, $$T = \left\{g_{i,j}
\in T_1 \mid \bar{i} = \bar{j} \right\},$$ and every $\mu \in \Irr
T$ can be obtained by restricting some $\lam_{s,t}$ to $T$.

We claim that $(\lam_{s,t})_T = (\lam_{s',t'})_T$ precisely when
there exist $x,y \in \ZZ$ such that $$s' = s+q_1x,~~~t' =
t+q_2y,~~~(q+1)|(x+y).$$
Indeed,  assume that $\lam_{s,t}|_T =
\lam_{s',t'}|_T$. Evaluating it at $g_{0,q+1} \in T$, we get
$\be^{(t'-t)(q+1)} = 1$, whence $t' = t+q_2y$ for some $y \in
\ZZ$. Similarly, by evaluating at $g_{q+1,0} \in T$ we get $s' =
s+q_1x$ for some $x \in \ZZ$. Finally, by evaluating at $g_{1,1}
\in T$ we obtain $1 = \al^{s'-s}\be^{t'-t} = \ga^{x+y}$ and so
$q+1$ divides $ x+y$. It is easy to check that the converse of our
claim holds.

Since $T_1 = T \times Z$, each $\mu \in \Irr T$ has precisely
$q+1$ extensions $\lam = \lam_{s,t}$ to $T_1$, which are uniquely
determined by their restrictions to $Z$, i.e. by $\overline{s+t}$.

\smallskip
4) Recall that $G = SU_n(q)$ (with $n \geq 3$)  has exactly $q+1$
Weil \ir characters, which can be obtained by restricting
$\om^{k}_{n,q}$, $0 \leq k \leq q$, to $G$.

Now suppose that $\mu \in \Irr T$ does not enter $\om^k_{n,q}|_T$
for some $k$, $0 \leq k \leq q$. By the previous observation, we
can find an extension $\lam_{s,t}$ of $\mu$ so that $k =
\overline{s+t}$. By the assumption, $\lam_{s,t}$ cannot enter
$\om^k_{n,q}|_T$. This implies by the conclusion of 2) that $t =
0$.

Conversely, suppose that $\mu = \lam_{s,0}|_T$ for some $s$. We
claim that $\mu$ is a constituent of $\om^l_{n,q}|_T$ if and only
if $l \neq \bar{s}$. Indeed, let $k := \bar{s}$. Now if $\mu$
enters $\om^k_{n,q}|_T$, then by the conclusion of 2) we must have
that $\mu = \lam_{u,v}|_T$ for some $u,v \in \ZZ$ with $k =
\overline{u+v}$ and $v \neq 0$. Thus, $\lam_{s,0}$ and
$\lam_{u,v}$ are two extensions to $T_1$ of $\mu$ with $\bar{s} =
k = \overline{u+v}$. By the last observation in 3), these two
extensions are the same, whence $v = 0$, a contradiction. Next we
consider any $l \neq k$, $0 \leq l \leq q$. Again by the last
observation in 3) we can find an extension $\lam_{s',t'}$ of $\mu$
to $T_1$ with $l = \overline{s'+t'}$. It follows by the
discussions in 3) that $s' = s+q_1x$, $t' = q_2y$, and $(q+1)$
divides $ (x+y)$. Notice that $q_1 \equiv -2 (\mod (q+1))$ and
$q_2 \equiv n-2 (\mod (q+1))$. Hence $$l-k \equiv (s'+t')-s =
q_2y+q_1x \equiv ny (\mod (q+1)).$$ Since $l \not\equiv k (\mod
(q+1))$ and $(n,q+1) = 1$, we must have that $(q+1)\not{|}y$ and
so $t' \neq 0$. Applying the results of 2), we see that
$\lam_{s',t'}$ is a constituent of $\om^l_{n,q}|_{T_1}$, and so
$\mu$ is a constituent of $\om^l_{n,q}|_T$.

To complete the proof of (iii), observe that the $q^2-1$
characters $\lam_{s,0}$ of $T_1$ have pairwise distinct
restrictions to $T$. Indeed, suppose $\lam_{s,0}|_T =
\lam_{s',0}|_T$. Then by 3), $s'-s = q_1x$ with $(q+1)$ divides $
x$, whence $s' = s$.

To obtain (iv), just notice that $1_T = \lam_{0,0}|_T$, and so the
only irreducible character of $G$ which is not contained in
$\Ind^G_T(1_T)$ is $\om^{0}_{n,q}$.

\medskip
5) To prove (ii), consider any Weil character $\chi = \om^{k}_{n,q}|_G$ of
$G$. We have shown in 4) that $\mu \in \Irr T$ does not enter
$\chi_T$ precisely when $\mu = \lam_{s,0}|_T$ with $k = \bar{s}$. Since
$0 \leq s \leq q^2-2$, there are exactly $q-1$ possibilities for $s$, and the
corresponding $q-1$ characters have pairwise distinct restrictions to $T$
by the previous paragraph.
\end{proof}

\begin{remar}
{\em   In Theorem \ref{complex-su}(iii) (in the notation of its
proof), $\Ind^G_T(\mu)$ can miss some nontrivial irreducible
character of $G$ precisely when $\mu = \lam_{s,0}|_T$ for some $0
\leq s \leq q^2-2$. }
\end{remar}

\section{Exceptional groups of Lie type}\label{exceptional}
We begin with the following observation:

\begin{lem}\label{regular}
Let $\GC$ be a connected reductive simply connected
algebraic group in characteristic $p$ and
let $G := \GC^{Fr}$. Let $\TC$ be an $Fr$-stable maximal torus in $\GC$ and let
$T := \TC^{Fr}$. Let $\chi \in \Irr(G)$ lie above the irreducible character $\alpha$
of $Z(G)$. Suppose that every element in $T \setminus Z(G)$ is regular and
that $\chi(1) \cdot |Z(G)| \geq |T|^{3/2}$. Then
$\chi|_T$ contains every irreducible character of $T$ that lies above $\alpha$.
\end{lem}

\begin{proof}
Consider any $\lambda \in \Irr(T)$ lying above $\alpha$ and any
$g \in T \setminus Z(G)$. By the assumption, $C_{\GC}(g)$ is a torus containing
$\TC$, hence $C_G(g) = T$. It is well known that $Z(G) \leq T$.
Now the orthogonality relations imply that
$|\chi(g)| \leq |T|^{1/2}$. It follows that
$$\begin{array}{ll}
  |T| \cdot |[\chi|_T,\lambda]_T| & = |\sum_{g \in T}\chi(g)\bar\lambda(g)|
  \geq |\sum_{g \in Z(G)}\chi(g)\bar\lambda(g)| -
       |\sum_{g \in T \setminus Z(G)}\chi(g)\bar\lambda(g)| \\ \\
  & \geq \chi(1) \cdot |Z(G)| - (|T|-|Z(G)|)|T|^{1/2} >
  \chi(1)\cdot |Z(G)| - |T|^{3/2} \geq 0.\end{array}$$
\end{proof}

In what follows, by {\it a finite exceptional group of Lie type of simply
connected type} we mean any {\it quasisimple} group $G = \GC^{Fr}$ of type
$G_2$, $^2G_2$, $^2B_2$, $^3D_4$, $F_4$, $^2F_4$, $E_6$, $^2E_6$, $E_7$, or $E_8$,
where $\GC$ is simply connected. This excludes the solvable cases, as well as
${\!^2G}_2(3)$, $G_2(2)$, and ${\!^2F}_4(2)$ which are not perfect.

\begin{thm}\label{exc}
Let $G$ be a finite exceptional group of Lie type of simply connected type.
Then $G$ contains a cyclic maximal torus $T$ such that
the following statements hold.

\medskip
{\rm (i)} For any $\alpha \in \Irr(Z(G))$
and any non-principal $\chi \in \Irr(G)$ lying above $\alpha$, $\chi|_T$ contains
every irreducible character of $T$ that lies above $\alpha$. In particular,
if $\vartheta \in \Irr(G/Z(G))$ is non-principal, then $\chi|_T$ contains
$1_T$.

\medskip
{\rm (ii)} There is some $s \in T$ such that $C_{G/Z(G)}(sZ(G)) = T/Z(G)$;
in particular, $T/Z(G)$ is self-centralizing in $G/Z(G)$.
\end{thm}

\begin{proof}
1) First we consider the case $G = E_6(q)_{sc}$
and choose $T = \GC^F$ to be a maximal torus
of order $\Phi_9(q)$, where $\Phi_m(q)$ denotes the $m^{\mathrm {th}}$ cyclotomic
polynomial in $q$. Observe that any $g \in T \setminus Z(G)$ is regular. Indeed,
since $\GC$ is simply connected, $C_{\GC}(g)$ is connected, cf.
\cite[Theorem 3.5.6]{C}. Write $q = p^f$, where $p$ is a prime. Then, since
$g \notin Z(G)$, $C_{G}(g) < G$. Furthermore, $C_G(g) \geq T$ has order
divisible by a {\it primitive prime divisor} $\ell$
of $p^{9f}-1$, i.e. a prime that does not divide $\prod^{9f-1}_{i=1}(p^i-1)$,
cf. \cite{Zs}.
 Using the description of the centralizers of semisimple elements in
$\GC$ given in \cite{Der} (and noting that centralizers of type
$SL_3(q^3)$ do not occur since $\GC$ is simply connected), we see
that $C_{\GC}(g)$ is a torus, whence $g$ is regular and $C_G(g) =
T$. Moreover, if we choose $g \in T$ to be of order $\ell$, then
since $\ell$ is coprime to $|Z(G)| = \gcd(3,q-1)$), we obtain that
$$C_{G/Z(G)}(gZ(G)) = C_G(g)/Z(G) = T/Z(G).$$ Furthermore,
$\chi(1)>|T|^{3/2}$ for any non-trivial $\chi \in \Irr(G)$ has
degree  by the Landazuri-Seitz-Zalesskii bounds \cite{LS},
\cite{SeZ} (and their improvements as recorded in \cite{T}). The
asserion follows by Lemma \ref{regular}.

The same argument applies to the case $G = {\!^2E}_6(q)_{sc}$ if
we choose $|T| = \Phi_{18}(q)$ (note that centralizers of type
$SU_3(q^3)$ do not occur since $\GC$ is simply connected, cf.
\cite{DerL}.) If $G = F_4(q)$, or $E_8(q)$, then, similarly,  we
choose $T$ of order $\Phi_{12}(q)$ or $\Phi_{30}(q)$,
respectively, and argue as above using \cite{Der} in the case of
$F_4(q)$ and \cite{LSS} in the case of $E_8(q)$ (which classifies
maximal subgroups of maximal rank in $E_8(q)$).

Suppose $G = {\!^2F}_4(q)$ with $q = 2^{2a+1} \geq 8$, or ${\!^2G}_2(q)$ with
$q = 3^{2a+1} \geq 27$. Then we choose $T$ of
order $q^2+q+1+(q+1)\sqrt{2q}$, respectively $q+\sqrt{3q}+1$, and argue as above
using \cite{DerL}.

\smallskip
2) Suppose $G = G_2(q)$ with $q \geq 3$ and $q \not\equiv -1 (\mod 3)$.
Then we can choose $T$ of order $\Phi_6(q)$ and argue as above (noting that
centralizers of type $SU_3(q)$ do not occur under our hypothesis on $q$).
Next suppose $G = G_2(q)$ with $q \geq 5$ and $q \equiv -1 (\mod 3)$.
Choosing $T$ of order $\Phi_3(q)$ and arguing as above (noting that
centralizers of type $SL_3(q)$ do not occur under our hypothesis on $q$),
we see that any element $g \in T \setminus \{1\}$ is regular with
$C_G(g) = T$, and so again we
are done by Lemma \ref{regular} if $\chi(1) \geq (q^2+q+1)^{3/2}$. If
$\chi(1) < (q^2+q+1)^{3/2}$, then in fact $\chi(1) = q^3+1$ and $|\chi(g)| = 1$
for all $1 \neq g \in T$ (see e.g. \cite[Anhang B]{Hiss}),
hence the proof of Lemma \ref{regular} yields the claim.

Consider the case $G = {\!^3D}_4(q)$.
Choosing $T$ of order $\Phi_{12}(q)$ and arguing as above using \cite{DerL},
we see that any element $g \in T \setminus \{1\}$ is regular with
$C_G(g) = T$, and so again we
apply Lemma \ref{regular}, provided that $\chi(1) \geq (q^4-q^2+1)^{3/2}$. If
$\chi(1) < (q^4-q^2+1)^{3/2}$, then by \cite{DMi} in fact $\chi(1) = q(q^4-q^2+1)$.
Now if $1 \neq g \in T$, then $g$ is $r$-singular for some prime $r$ dividing
$\chi(1)$ and $\chi$ has $r$-defect $0$, whence $\chi(g) = 0$.
Thus the claim follows from the proof of Lemma \ref{regular}.

Suppose $G = {\!^2B}_2(q)$ with $q = 2^{2a+1} \geq 8$, then we choose
$T$ of order $q+\sqrt{2q}+1$. Then for any $1 \neq g \in T$ and any
$1_G \neq \chi \in \Irr(G)$, $C_G(g) = T$, $|\chi(g)| \leq 1$ but $\chi(1) > |T|$
(see e.g. \cite{Burk}). Hence the proof of Lemma \ref{regular} yields the claim.

\smallskip
3) Finally, we consider the case $G = E_7(q)_{sc}$. Note that
$G/Z(G)$ contains a subgroup $S \cong PSL_2(q^7)$, cf. \cite[Table
5.1]{LSS}. We claim that $E_7(q)_{sc}$ contains a subgroup $L\cong
SL_2(q^7)$. For this it suffices to consider $q$ odd. We can view
$(E_7)_{sc}$ as a component of the centralizer of an involution in
$E_8$. The subsystem $A_1^7$ of $E_7$ is described for example in
Lemma 2.1 of \cite{LiSe}. It is shown there that the subgroup $H$
with $H^0 = A_1^7$ acts irreducibly on the $56$-dimensional module
of $E_7$, the central involution acting there as $-\Id$. It
follows that $SL_2(q^7)$ is a subgroup of $E_7(q)_{sc}$, as
claimed. Notice that $G$ contains a maximal torus $T$ of order
$q^7-1$ which contains a Sylow $\ell$-subgroup for some primitive
prime divisor $\ell$ of $q^7-1$. We may assume that an element $s
\in T$ of order $\ell$ is contained in $L$. On the one hand, using
\cite{FJ} we see that $s$ is regular and $C_G(s) = T$. Since $\ell
> 2 \geq |Z(G)| = \gcd(2,q-1)$, we have $$C_{G/Z(G)}(sZ(G)) =
C_G(s)/Z(G) = T/Z(G).$$ On the other hand, $|C_L(g)| = q^7-1$.
Thus $T$ can be embedded in $L$ as a maximal torus. Now if $\alpha
= 1_{Z(G)}$, then $\chi|_L$ contains a faithful irreducible
character $\varsigma$ of $L/Z(G)$. If $Z(G) \neq 1$ and $\alpha
\neq 1_{Z(G)}$, then $\chi|_L$ contains a faithful irreducible
character $\varsigma$ of $L$. In either case, by Proposition
\ref{ca2}, we conclude that $\varsigma|_T$ contains all
irreducible characters of $T$ lying above $\alpha$.
\end{proof}

\begin{remar}
{\em In the case where 
 $G\in \{E_8(q), F_4(q)$, ${\!^2F}_4(q), G_2(q)\}$, there is
an alternative (and perhaps more conceptual) way to prove the result. As
is observed in \cite{z89}, see also \cite{GT}, for every torus $T$
of $G$ and every complex \irr $\phi$ of $G$ the restriction $\phi|_T$
contains $1_T$. This follows by taking a reduction of $\phi$
modulo $p$ (in the sense of Brauer) and observation that every restricted \ir
$p$-modular \rep of $\GC$ has weight 0. Note, however, that this
method cannot be used when $\al$ is non-trivial.}
\end{remar}

In the next statement, we also include
$G_2(2)'$, ${^2G}_2(3)'$,and ${\!^2F}_4(2)'$ among
the simple exceptional groups of Lie type.
In the last two cases,
$G = {^2G}_2(3)'$ and $G = {\!^2F}_4(2)'$,
by a Steinberg character of $G$ we understand any of the irreducible
constituents of the groups  $G = {^2G}_2(3)$ and $G = {\!^2F}_4(2)$,
of degree and $9$ and $2048$, respectively.

\begin{cor}\label{exc1}
Theorems $\ref{main1}$, $\ref{main2}$, and $\ref{mn3}$ hold for
simple exceptional groups of Lie type, except for $G_2(2)' \cong
SU_3(3)$.
\end{cor}

\begin{proof}
By Theorem \ref{exc} and Corollary \ref{cck} (below), this is true unless
$G$ is one of the following groups ${^2G}_2(3)'$,
$G_2(2)'$, ${\!^2F}_4(2)'$.
In the first case, we have an exception.
In the second case, ${^2G}_2(3)'\cong SL_2(8)$, and Theorem
\ref{main1} follows by previous result for $SL_2(8)$, and
Theorem \ref{main2} holds by direct computation.
In the third case,
\ref{main1} holds by choosing $T = C_G(s)$ of order $13$, and
Theorem \ref{main2} holds by direct computation (or by applying
\cite[Lemma 3.1]{HZ}).
\end{proof}

\bigskip
{\bf Proof of Theorem \ref{mn2}.}
Obviously, $\Delta_G\subseteq
\Delta_C$. As the group algebra ${\mathbb C}G$ is a direct sum of
simple rings, so are $\Delta_G$ and $ \Delta_C$, and hence
$\Delta_G$ is an ideal of $ \Delta_C$. There is a natural
bijection between simple rings in question and $\Irr G$. It
follows the direct summands of $\Delta_G$, resp.,  $ \Delta_C$
corresponds to the \ir \reps of $G$ that do not occur in $\Pi_G$
resp., $\Pi_C$. By Theorem \ref{mn3}, $ \Delta_C=0$, unless
$G\cong PSU_n(q)$ with $(2(q+1),n)=1$, so the assertion follows
in the non-exceptional case.
In the exceptional case there is a single
\rep of $G$ not occurring in $\Pi_C$. Therefore, $ \Delta_C$ is a
simple ring, so either $\Delta_G=0$ or $\Delta_G= \Delta_C$.
However, $\Delta_G\neq 0$ by \cite[Theorem 1.1]{HZ}. So the result
follows.

\section{The square of the Steinberg character}

\medskip
Next we turn our attention to the square of the Steinberg
character $\Ste$ of a simple group $G$ of Lie type.
Such a group can be obtained from a simple simply
connected algebraic group $\GC$ as $L/Z(L)$, where $L = \GC^{Fr}$
is the fixed point subgroup of a Frobenius endomorphism $Fr:\GC
\to \GC$ and $L$ is quasisimple. Let $\TC$ be an $Fr$-stable
maximal torus in $\GC$. Then $T = \TC\cap L$ is called a maximal
torus of $L$. Set $W(T)=(N_\GC(\TC)/\TC)^{Fr}$.

\begin{lem}\label{3kk}
Let $\chi$ be an irreducible character of $L = \GC^{Fr}$. Then
$$[\chi\cdot \Ste, \Ste]_G=\sum _{(\TC  )}\frac{[\chi |_T ,1_T
]_T}{|W(T)|},$$ where the sum ranges over representatives of the
$G$-conjugacy classes of $F$-stable maximal tori $\TC$ of $\GC$
and  $T=\TC^{Fr}$.
\end{lem}

\begin{proof}
 By \cite[7.15.2]{DL}, $\Ste\cdot \Ste=\sum _{(\TC
)}\frac{1}{|W(T)|}1_T^G,$ where the sum ranges as above. Note that
$[\chi, 1_T^G]_G=[\chi|_T,1_T]_T$ by the Frobenius reciprocity.
This implies the lemma.
\end{proof}

\begin{cor}\label{cck}
Let $G$ be a simple group of Lie type,
let $\Ste$ be the Steinberg character of $G$, and let $\chi$ be
any \ir character of $G$.
Then $\chi$ is a constituent of the tensor square $\Ste^2$
if and only if for some maximal torus $T$ of $G$, the restriction
of $\chi$ to $T$ involves $1_T$.
\end{cor}

\begin{proof}
Again view $G$ as $L/Z(L)$ for some $L = \GC^F$ as above.
Since the Steinberg character of $L$ is trivial at $Z(L)$, we may replace
$G$ by $L$ and $\chi$ by any irreducible character of $L$ trivial at $Z(L)$.
As $\Ste$ is self-dual, we have
$[\chi, \Ste\cdot \Ste]_G=[\chi\cdot \Ste, \Ste]_G$.
Now we apply Lemma \ref{3kk} and note that all terms on the right hand side of
the formula are non-negative. Therefore, $[\chi\cdot \Ste, \Ste]_G>0$ whenever
there exists a maximal torus $T$ of $L$ such that $[\chi|_T,1_T]_T>0$.
\end{proof}

The assertion of Corollary \ref{cck} holds also for the groups
$Sp_4(2)'$, ${\!^2G}_2(3)'$, ${\!^2F}_4(2)'$, with
the Steinberg character and torus suitably defined, but
not for the group $G_2(2)'$, where the irreducible characters of degree $32$
are not real.

\medskip
One might think that in the exceptional case $G = SU_n(q)$ with
$(n,2(q+1)) = 1$ in Theorem \ref{main2}, one could try to replace
$\Ste^2$ by $\tau^2$ or $\tau\bar\tau$ for some other $\tau \in
\Irr G$ so that the resulting character would include every $\chi
\in \Irr (G)$ as an irreducible character. It is however not the
case.

\begin{lem}\label{other}
 Let $G=SU_n(q)$, $n \geq 3$, and $
 (n,2(q+1))=1$. Let $\tau$ be an
arbitrary complex irreducible character of $G$ and
let $\phi_{\min}$ be the \ir non-trivial
character of minimum degree $(q^n-q)/(q+1)$.  Then no
$\varrho \in \{\tau^2,\tau\bar\tau\}$ can contain both  $1_G$ and $\phi_{min}$
as irreducible constituents.
\end{lem}

\begin{proof}
Note that the permutation character $\pi$ of the conjugation
action is $\sum_{\chi \in \Irr (G)}\chi\bar\chi$ and it does not
contain $\phi_{min}$ by Theorem \ref{main1}. Hence the assertion
follows if $\varrho = \tau\bar\tau,$ or if $\varrho = \tau^2$ and
$\tau = \bar\tau$. If $\tau \neq \bar\tau$, then
$[1_G,\tau^2]_G=[\bar\tau,\tau]_G=0$.
\end{proof}

\medskip
{\bf Proof of Theorem \ref{main2}}. This follows from Theorem \ref{main1} and
Corollary \ref{cck}.

\affiliationone{G. Heide
\email{gerhardheide@yahoo.co.uk}}
\affiliationtwo{Jan Saxl\\
  Department of Pure Mathematics and Mathematical Statistics\\
  University of Cambridge\\ 
  Cambridge CB3 0WA\\ 
  UK
\email{J.Saxl@dpmms.cam.ac.uk}}
\affiliationthree{Pham Huu Tiep\\
  Department of Mathematics\\ 
  University of Arizona\\ 
  Tucson, AZ 85721\\ 
  USA
\email{tiep@math.arizona.edu}}
\affiliationfour{Alexandre E. Zalesski\\
  Dipartimento di Mathematica e Applicazioni\\ 
  Universit\`a degli Studi di Milano-Bicocca\\ 
  20125 Milano\\ 
  Italy
\email{alexandre.zalesskii@gmail.com}}

\end{document}